\documentclass[12pt,draftcls,onecolumn]{IEEEtran}
\usepackage[ruled,boxed]{algorithm2e}

\usepackage{amsmath,graphicx,color,psfrag,mathrsfs,bm}
\usepackage{hyperref}
\usepackage{pstool}
\usepackage{longtable}          
\usepackage{amsthm}
\usepackage{amssymb}

\usepackage{multicol}
\newtheorem{thm}{Theorem}[section]

\newtheorem{prop}[thm]{Proposition}

\theoremstyle{definition}
\newtheorem{defn}{Definition}[section]
\theoremstyle{remark}
\newtheorem{rem}{Remark}[section]

\begin{document}

\title{Asymptotically Minimax Robust Likelihood Ratio Test}

\author{%
  \IEEEauthorblockN{G\"okhan  G\"ul}\\
  \IEEEauthorblockA{Preventive Cardiology and Preventive Medicine, Department of Cardiology, University Medical Center of the Johannes Gutenberg University Mainz\\
                    Clinical Epidemiology and Systems Medicine, Center for Thrombosis and Hemostasis, University Medical Center  Johannes Gutenberg University Mainz\\
                    German Center for Cardiovascular Research (DZHK), Partner Site Rhine Main, University Medical Center of the Johannes Gutenberg University Mainz\\
                    Langenbeckstra\ss e 1, 55131 Mainz, Germany\\
                    Email: goekhan.guel@unimedizin-mainz.de}
}

\maketitle

\begin{abstract}
This paper develops a unified framework for asymptotically minimax robust hypothesis testing under distributional uncertainty, applicable to both Bayesian and Neyman--Pearson formulations (Type-I and Type-II). Uncertainty classes based on the KL-divergence, $\alpha$-divergence, and its symmetrized variant are considered. Using Sion's minimax theorem and Karush-Kuhn-Tucker conditions, the existence and uniqueness of the resulting robust tests are established. The least favorable distributions and corresponding robust likelihood ratio functions are derived in closed parametric forms, enabling computation via systems of nonlinear equations. It is proven that Dabak’s approach does not yield an asymptotically minimax robust test. The proposed theory generalizes earlier work by offering a more systematic and comprehensive derivation of robust tests. Numerical simulations confirm the theoretical results and illustrate the behavior of the derived robust tests.
\end{abstract}

\begin{IEEEkeywords}
Hypothesis testing, event detection, robustness, least favorable distributions, minimax optimization.
\end{IEEEkeywords}

\IEEEpeerreviewmaketitle

\section{Introduction}
The design of optimal hypothesis tests relies on the assumption that the statistical models of the observed data are perfectly known~\cite{Kay1998}. However, complete knowledge of the underlying distributions is often unavailable. Model mismatches may arise from limited training data, non-stationary environments, or imperfect physical models, potentially leading to significant deviations from nominal performance. To address these challenges, robust hypothesis testing was developed to ensure reliable decision-making even when the true distributions differ from their nominal descriptions~\cite{Levy2008}. Early contributions by Huber and Strassen established a principled framework that models distributional uncertainty through sets of probability measures capturing admissible deviations from the nominal models~\cite{Huber1965,HuberStrassen1973}. This formulation established the basis of the minimax robust approach, where the focus shifts from optimizing performance under a single model to safeguarding it against the worst admissible alternatives.\\
In the minimax robust framework, the true distributions under each hypothesis are assumed to belong to uncertainty classes, typically defined by statistical neighborhoods or divergence-based constraints~\cite{Huber1965,levy09}. The aim is to minimize the maximum achievable risk, leading to decision rules that remain optimal in the worst case. The resulting tests are likelihood ratio tests constructed from the least favorable distributions (LFDs)\footnote{For brevity, no notational distinction is made between least favorable distributions and their corresponding densities unless confusion may arise.}, which represent the extremal elements of the uncertainty classes. Such formulations provide guaranteed performance under bounded model perturbations and represent a central concept in robust statistical inference.\\
Although guaranteed performance is of central interest, the existence of minimax robust tests is not guaranteed and is critically dependent on the choice of the uncertainty classes. When no minimax test exists within the class of deterministic decision rules, it may still exist among randomized decision rules \cite{gul7}. A notable limitation, however, is that such tests are usually minimax robust only for a single sample and cannot be extended to multiple samples while preserving robustness \cite{gul6}. In the absence of fixed-sample minimax robust tests, a natural alternative is to consider asymptotically minimax robust tests, which minimize the exponential decay rate of the error probability as the number of samples increases.

\subsection{Related Work}
Minimax robust hypothesis testing was grounded in the foundational analyses of Huber and Strassen. Huber developed the minimax formulation for the $\epsilon$-contamination and total variation neighborhoods, identifying the least favorable distributions and proving that the corresponding clipped likelihood ratio test is minimax robust for any finite sample size \cite{Huber1965}. Subsequent generalizations to $2$-alternating capacities \cite{HuberStrassen1973} broadened the theoretical scope of minimax robustness beyond contamination-type models. Building on these developments, the framework was extended to alternative formulations of uncertainty grounded in information-theoretic divergence measures such as the KL-divergence \cite{gul6}.\\
Dabak and Johnson observed that excluding non-smooth distributions from the $\epsilon$-contamination model could yield uncertainty models better suited for practical scenarios, particularly where modeling errors are a primary concern. Building on this idea, they proposed using the KL-divergence to define uncertainty classes and derived the corresponding robust test for the asymptotic case, i.e., as the number of measurements tends to infinity \cite{dabak}. Under several assumptions, Levy demonstrated that a single-sample minimax robust test could be designed for the same uncertainty model, provided that randomized decision rules are permitted. Subsequently, all of Levy's assumptions were removed in a later work \cite{gul6}.\\
A key limitation of the KL-divergence formulation was that both the divergence measure and the a priori probabilities of the hypotheses were designed to be fixed and non-adjustable \cite{gul6}. By replacing the KL-divergence with the $\alpha$-divergence, these constraints were lifted \cite{gul7}. Remarkably, for the entire $\alpha$-divergence neighborhood and for any a priori probabilities of the hypotheses, the corresponding minimax robust test was a censored likelihood ratio test with a well-defined randomization function \cite{gul7}.\\
From a broader perspective, recent research in robust hypothesis testing has evolved alongside the field of distributionally robust optimization~\cite{RahimianMehrotra2019}. Various formulations have been proposed to capture distributional uncertainty, including Wasserstein-based models~\cite{GaoXie2018}, Sinkhorn-type extensions~\cite{WangGaoXie2024}, and kernel-based uncertainty sets~\cite{SunZou2022,SchrabKim2024}. These approaches provide convex or approximately tractable formulations with strong analytical guarantees. Moment-constrained uncertainty classes~\cite{MageshSunVeeravalliZou2024} and empirical distribution-based approaches~\cite{WangGaoXie2022} have further generalized robust testing to handle data-driven settings and allow for practical finite-sample implementations. Additional studies have explored decentralized~\cite{gul5}, sequential~\cite{CaoEtAl2022}, and adversarially robust~\cite{PuranikMadhowPedarsani2021} settings. Alternative divergence measures, such as those based on the Hellinger distance, have also been investigated to improve robustness to distributional misspecification while retaining finite-sample guarantees~\cite{guel_helliger,SureshEtAl2021}.

\subsection{Motivation}
\begin{enumerate}
  \item The derivations in \cite{dabak, dabak2}, which are later summarized in \cite{levy}, do not yield asymptotically minimax robust Neyman-Pearson (NP) tests. This necessitates explicit derivations and analysis to achieve minimax robustness.
  \item The theoretical designs for the asymptotic case consider NP-formulations by default, probably because they result in simpler solutions \cite{dabak,moment}. However, as shown by Chernoff, NP tests exhibit the slowest exponential decay of the error probability \cite{chernoff}. Therefore, it is necessary to obtain the asymptotically minimax robust tests for the fastest decay rate of the error probability.
\end{enumerate}

\subsection{Summary of the paper and its contributions}
In this paper, a formal framework is developed for designing asymptotically minimax robust binary hypothesis tests across various uncertainty classes. The existence and uniqueness of minimax robust tests are established in general. Considering the Karush-Kuhn-Tucker (KKT) conditions, the least favorable distributions and the robust likelihood ratio functions (LRFs) are derived in parametric forms, which can be made explicit by solving a set of non-linear equations. In the sequel, the contributions of this paper together with their relation to prior works are summarized.
\begin{enumerate} 
  \item It is shown that any asymptotically minimax robust test can be designed via solving
  \begin{equation*}
    \min_{u\in(0,1)} \max_{(G_0,G_1)\in\mathscr{G}_0\times \mathscr{G}_1}D_u(G_0,G_1),
  \end{equation*}
where
\begin{equation*}
D_u(G_0,G_1)=\int_{\Omega}{g_1}^u {g_0}^{1-u}d\mu
\end{equation*}
is denoted as the $u$-affinity (see Section~\ref{sec_rm}).
\item For uncertainty classes based on the KL-divergence, $\alpha$-divergence and the symmetrized $\alpha$-divergence the LFDs of the asymptotically minimax robust \emph{Bayesian} tests are obtained in parametric forms (see Theorem~\ref{theorem02}, Theorem~\ref{theorem04} and Theorem~\ref{theorem04s}).
\item For the KL-divergence neighborhood, the LFDs of the asymptotically minimax robust \text{\emph{NP-tests}} of \text{Type-I} and \text{Type-II} are obtained in parametric forms. The parameters of LFDs can be found by solving four non-linear coupled equations (see \eqref{59x1}) and the corresponding test is different from the ones derived in \cite{dabak,dabak2,levy} (see Theorem~\ref{theorem03np} and Remark~\ref{theorem03np}).
\end{enumerate}
These results provide a systematic framework for designing robust hypothesis tests in both Bayesian and Neyman-Pearson settings under general uncertainty classes.

\subsection{Outline of the paper}
The rest of the paper is organized as follows. In Section~\ref{sec3}, asymptotically minimax robustness is introduced. In Section~\ref{sec5}, the equations formulating asymptotic minimax robustness both in Bayesian as well as in Neyman-Pearson sense are derived, saddle value condition is characterized and the problem statement is made. In Section~\ref{sec6}, the least favorable distributions and asymptotically minimax robust tests are obtained for the uncertainty classes based on the KL-divergence, $\alpha$-divergence and the symmetrized $\alpha$-divergence. In Section~\ref{sec8}, simulations are performed to evaluate and exemplify the theoretical derivations. Finally in Section~\ref{sec9}, the paper is concluded.

\subsection{Notations}
The following notations are applied throughout the paper. Upper case symbols are used for probability distributions and random variables, and the corresponding lower case symbols denote the density functions and observations, respectively. Boldface symbols are used for the sequence of random variables, sequence of observations or joint functions. The hypotheses $\mathcal{H}_0$ and $\mathcal{H}_1$ are associated with the nominal probability measures $F_0$ and $F_1$, whereas the corresponding actual distributions are denoted by $G_0$ and $G_1$. The sets of probability distributions are denoted by $\mathscr G_0$ and $\mathscr G_1$. Every probability measure, e.g. $G[\cdot]$, is associated with its distribution function $G(\cdot)$ and the density function $g(\cdot)$. The notation $\hat{(\cdot)}$ indicates the least favorable distributions $\hat {G}_j\in \mathscr G_j$, the corresponding densities $\hat{g}_j$, or the robust likelihood ratio test $\hat{\delta}$. The expected value of a random variable $Y\sim G_j$ is denoted by $\mathbb{E}_{G_j}[Y]$. The argument (value on the domain) of the subsequent operation is denoted by $\arg$.

\section{Asymptotic Minimax Robustness}\label{sec3}
Let $\boldsymbol{Y}=(Y_1,\ldots,Y_n)$ be a sequence of $n$ independent and identically distributed (i.i.d.) random variables (r.v.s), each taking values on $\Omega$. The distribution of $Y_k$ is denoted by $G_j$ under the hypothesis $\mathcal{H}_j$ and is not known exactly but belongs to an uncertainty class $\mathscr{G}_j$. Given $\boldsymbol{Y}$, the goal is to decide which of the following hypothesis is true
\begin{align}
\mathcal{H}_0&: Y_k \sim G_0,\quad  G_0\in\mathscr{G}_0,\nonumber\\
\mathcal{H}_1&: Y_k \sim G_1,\quad  G_1\in\mathscr{G}_1,
\end{align}
where $k\in \{1,\ldots,n\}$. Let $P_F=G_0[\delta(\boldsymbol{Y})=1]$ be the false alarm-, $P_M=G_1[\delta(\boldsymbol{Y})=0]$ be the miss detection probability and $\pi_0=P(\mathcal{H}_0)$ be the a-priori probability of the hypothesis $\mathcal{H}_0$. Furthermore let $\delta:\boldsymbol{Y}\mapsto\{0,1\}$ be the statistical test and
\begin{equation}\label{eq2}
P_E(\delta,G_0,G_1)=\pi_0 P_F(\delta,G_0) + (1-\pi_0)P_M(\delta,G_1)
\end{equation}
define the overall error probability. Then, we are interested in finding a test $\hat{\delta}$ together with the least favourable distributions $\hat{G}_0$ and $\hat{G}_1$, which solve the minimax optimization problem
\begin{equation}\label{eq10}
\min_{\delta}\max_{(G_0,G_1)\in {\mathscr{G}}_0\times{\mathscr{G}}_1}P_E(\delta,G_0,G_1).
\end{equation}
Let $\hat{l} = \hat{g}_1 / \hat{g}_0$, define $X_k = \log \hat{l}(Y_k)$, and let the empirical mean be denoted by $S_n(\boldsymbol{X}) = \frac{1}{n} \sum_{k=1}^n X_k$. Since the minimax problem \eqref{eq10} is solved by a (log-)likelihood ratio test corresponding to the least favourable distributions, the robust decision rule takes the form
\begin{equation}\label{eqgfg}
\hat{\delta}(\boldsymbol{X}) = \begin{cases} 
0, & S_n(\boldsymbol{X}) < t, \\
1, & S_n(\boldsymbol{X}) \geq t,
\end{cases}
\end{equation}
where $t \in \mathbb{R}$ is a decision threshold selected to minimize the worst-case error probability $P_E$.\\
In this work, attention is directed to the asymptotic regime where $n$ is sufficiently large, and in particular, the behavior as $n \to \infty$ is considered. In what follows, an extended version of Cramér’s theorem is stated to describe the asymptotic decay rates of the false alarm and miss detection probabilities \cite{Cramer}. These results will later serve as the basis for deriving minimax optimality conditions, which in turn enable the derivation of least favourable distributions in parametric forms.
\begin{thm}\label{theorem0}
Let $\boldsymbol{Y}=(Y_1,\ldots,Y_n)$ be a sequence of i.i.d. r.v.s, where each $Y_k$ is distributed as $G_j\in \mathscr{G}_j$ under $\mathcal{H}_j$. Furthermore, let each $X_k$ has a finite moment generating function
\begin{equation}\label{eq29}
M_{X_k}^j(u)=\mathbb{E}_{G_j}\left[\exp\left({u X_k}\right)\right]<\infty,\quad j\in\{0,1\}.
\end{equation}
For the test given by \eqref{eqgfg}, if
\begin{equation}\label{eq26}
\mathbb{E}_{G_0}[X_k]<t<\mathbb{E}_{G_1}[X_k],
\end{equation}
then, as $n\to\infty$, false alarm and miss detection probabilities decrease exponentially
\begin{align}\label{eq27}
\lim_{n\rightarrow \infty}\frac{1}{n}\log G_0[S_n(\boldsymbol{X})> t]=-I_0(t)\\\label{eq275}
\lim_{n\rightarrow \infty}\frac{1}{n}\log G_1[S_n(\boldsymbol{X})\leq t]=-I_1(t)
\end{align}
with the rate functions given by
\begin{equation}\label{eq28}
I_j(t)=\sup_{u\in\mathbb{R}}\left(tu-\log M_{X_k}^j(u)\right),\quad j\in\{0,1\}.
\end{equation}
\end{thm}

\begin{IEEEproof}
The results (\eqref{eq27} and \eqref{eq275}, respectively) follow by applying Cram\'{e}r's theorem \cite{Cramer} twice, first under $\mathcal{H}_0$ to the sequence of r.v.s $(X_k)_{k\geq 1}$, each satisfying $\mathbb{E}_{G_0}[X_k]<t$, and second under $\mathcal{H}_1$ to the sequence of r.v.s $(-X_k)_{k\geq 1}$, each satisfying $t<\mathbb{E}_{G_1}[X_k]$.
\end{IEEEproof}

Based on Theorem~\ref{theorem0} asymptotic minimax robustness can be defined as follows.

\begin{defn}\label{defn1}
A test of the form \eqref{eqgfg} is called asymptotically minimax robust for threshold \(t\) if,
\begin{equation*}
\lim_{n\rightarrow\infty}\frac{1}{n}\log  G_0\left[S_n(\boldsymbol{X})> t\right]\leq \lim_{n\rightarrow\infty}\frac{1}{n}\log  \hat{G}_0\left[S_n(\boldsymbol{X})> t\right]
\end{equation*}
and
\begin{equation*}
\lim_{n\rightarrow\infty}\frac{1}{n}\log  G_1\left[S_n(\boldsymbol{X})\leq t\right]\leq \lim_{n\rightarrow\infty}\frac{1}{n}\log  \hat{G}_1\left[S_n(\boldsymbol{X})\leq t\right]
\end{equation*}
hold for all \((G_0,G_1)\in\mathscr{G}_0\times\mathscr{G}_1\).
\end{defn}

\begin{rem}\label{rem:effective_classes}
Definition~\ref{defn1} implicitly relies on the regularity conditions \eqref{eq29} and \eqref{eq26}. These assumptions are mild and not restrictive. For instance, \eqref{eq29} holds whenever the uncertainty classes admit a common dominating measure and share a common support, while \eqref{eq26} can be ensured by choosing a suitable threshold (e.g., $t=0$). 
\end{rem}

\subsubsection{Optimum Threshold}\label{sec42}
An optimal value of $t$ which maximizes the asymptotic decrease rate of the error probability is of particular interest. This problem was first solved by Chernoff \cite{chernoff} for a fixed pair of distributions. In the following, a solution is provided, but without intially assuming that the data samples are obtained from $\hat{G}_0\in \mathscr{G}_0$ and $\hat{G}_1\in \mathscr{G}_1$, cf. Theorem~\ref{theorem0}. This perspective is essential for characterizing the asymptotic behaviour of the error exponents when the actual data-generating distributions do not coincide with the least favourable distributions of the minimax test.

\begin{thm}\label{theorem01}
The threshold that asymptotically minimizes the error probability is given by
\begin{equation}
\arg \min_{t} \lim_{n \to \infty} P_E(\hat\delta, G_0, G_1) = 0,
\end{equation}
if $G_0 := \hat{G}_0$ and $G_1 := \hat{G}_1$.
\end{thm}

A proof of Theorem~\ref{theorem01} is provided in Appendix~\ref{appendix3}. The derivations indicate that $t=0$ is guaranteed to be optimum, when the true data-generating distributions coincide with those used to construct the test, i.e., $G_j := \hat{G}_j$. In the presence of any model mismatch, however, the optimum thereshold is potentially nonzero.

\section{Derivation of Minimax Equations and Problem Formulation}\label{sec5}
In what follows, both Bayesian ($t = 0$) and Neyman–Pearson-type asymptotically minimax robust tests are derived. This broader formulation allows comparison with earlier results and provides a unified framework for characterizing asymptotic robustness. In particular, it highlights key differences from Dabak’s derivations \cite{dabak}.\\
Unless stated otherwise, asymptotically minimax robust test refers to the Bayesian test with threshold $t = 0$, consistent with the objective of minimizing error probability. While Definition~\ref{defn1} accommodates general thresholds to cover Neyman–Pearson-type tests, the truly minimax robust solution occurs at $t = 0$. Thus, the saddle value condition and problem formulation that follow focus on this case.

\subsection{Bayesian Tests}\label{sec_rm}
From the previous section the moment generating function $M_{X_k}^0(u)$ is equivalent to the \text{$u$-affinity} which is defined as
\begin{equation*}
D_u(G_0,G_1)=\int_{\Omega}{g_1}^u{g_0}^{1-u}d\mu,
\end{equation*}
where $g_j$ represents the density of $G_j$, $j\in\{0,1\}$. In fact $D_u$ is a non-scaled version of the Renyi divergence \cite{renyi}. Some important properties of the \text{$u$-affinity}, used in the subsequent derivations, are listed below:
\begin{enumerate}
  \item $D_u$ is non-negative, i.e. $D_u(G_0,G_1)\geq 0$
  \item $D_u$ is continuous and convex in $u$
  \item $D_u$ is continuous and jointly concave in $(g_0,g_1)$
  \item $D_{u=0}(G_0,G_1)=D_{u=1}(G_0,G_1)=1$
  \item $D_u(G_0, G_1) \in [0, 1]$, which follows from Property~1 and Hölder’s inequality~\cite{holder}.
  \item If $\hat{u} = \arg\min_u D_u(G_0, G_1)$, then $\hat{u} \in (0,1)$, by Properties~2 and~4.
  \item $D_u$ is related to the $\alpha$-divergence via $D_\alpha(G_0,G_1)=(1-D_u(G_0,G_1))/u(1-u)|_{u:=\alpha}$
\end{enumerate}
Properties~$1$ and~$4-7$ follow directly from the definition and basic inequalities. Properties~$2$ and~$3$ are inherited from the established properties of $\alpha$-divergence, using Property~$7$ \cite{entropy}.\\
The minimax robust test to be derived is a log-likelihood ratio test \eqref{eqgfg}, in which both the test and the data samples are constructed using the least favorable distributions $\hat{G}_0$ and $\hat{G}_1$. As shown in the proof of Theorem~\ref{theorem01}, this justifies the choice of threshold $t=0$. Hence, using the identities \text{$-\inf(x)=\sup(-x)$} for $x<0$ and \text{$\inf(-x)=-\sup(x)$} together with $t=0$, the minimax equations can be derived from \eqref{eq29} and \eqref{eq28} as
\begin{align}\label{eq37}
\hat{g}_0=&\arg\sup_{G_0\in \mathscr{G}_0}\left(\inf_{u_0\in \mathbb{R}}\log \int_{\Omega}\left(\frac{\hat{g}_1}{\hat{g}_0}\right)^{u_0} g_0 d\mu\right),\nonumber\\
\hat{g}_1=&\arg\sup_{G_1\in \mathscr{G}_1}\left(\inf_{u_1\in \mathbb{R}}\log \int_{\Omega}\left(\frac{\hat{g}_1}{\hat{g}_0}\right)^{u_1} g_1 d\mu\right).
\end{align}
These coupled equations are, in general, mathematically intractable, particularly when the uncertainty sets $\mathscr{G}_0$ and $\mathscr{G}_1$ are infinite-dimensional. However, if a solution to the supremum part of the minimax problem—i.e., the least favorable densities $\hat{g}_0$ and $\hat{g}_1$—is available, then the corresponding optimization over the decision rule can equivalently be expressed in terms of the $u$-affinity cost functions as
\begin{align}\label{eq38}
C_0(\hat{G}_0,\hat{G}_1;\hat{u}_0)=&\inf_{u_0\in \mathbb{R}}\log D_{u_0}(\hat{G}_0,\hat{G}_1) ,\nonumber\\
C_1(\hat{G}_0,\hat{G}_1;\hat{u}_1)=&\inf_{u_1\in \mathbb{R}}\log D_{1+u_1}(\hat{G}_0,\hat{G}_1).
\end{align}

\begin{prop}\label{prop3}
The solutions to both optimization problems in \eqref{eq38} coincide when $\hat{u}_0=-\hat{u}_1$, with $\hat{u}_0\in[0,1]$, i.e. $C_0(\hat{G}_0,\hat{G}_1;\hat{u}_0)=C_1(\hat{G}_0,\hat{G}_1;-\hat{u}_1)$.
\end{prop}

\begin{IEEEproof}
From the Property~$6$ of $D_u$, it follows that the infimum in \eqref{eq38} is attained at a unique minimizer $\hat{u}_0 \in [0,1]$. By analogous reasoning, the minimizer $-\hat{u}_1$ in the second equation of \eqref{eq38} also lies within the interval $[0,1]$; see also the structure of the coupled equations in \eqref{eq37}. Furthermore, since $-C_0$ corresponds to the Chernoff distance, which is symmetric with respect to its arguments \cite[p.~82]{levy}, we have
\begin{equation*}
C_0(\hat{G}_0,\hat{G}_1;\hat{u}_0) = C_0(\hat{G}_1,\hat{G}_0;\hat{u}_0) = C_1(\hat{G}_0,\hat{G}_1; -\hat{u}_1).
\end{equation*}
This completes the proof.
\end{IEEEproof}

Proposition~\ref{prop3} implies that both optimization problems are equivalent and have the same solution for $\hat{u}_0=-\hat{u}_1$. Hence, it is sufficient to consider only one of them. Considering the first formulation and removing the $\log$ term, the problem to be solved can be reduced to
\begin{equation}\label{eq42}
(\hat{g}_0,\hat{g}_1)=\arg \sup_{(G_0,G_1)\in \mathscr{G}_0\times \mathscr{G}_1} \inf_{u\in(0,1)}D_u(G_0,G_1).
\end{equation}
This can be done because for both $\sup$ and $\inf$ we have
\begin{equation*}
\frac{\partial \log D_u}{\partial u}=0\Longrightarrow \frac{\partial D_u}{\partial u}=0,
\end{equation*}
and $\log$ is an increasing mapping from $[0,1]$ to $\mathbb{R}_{\leq 0}$, see Property~$5$ of $D_u$.

\subsection{Neyman-Pearson Tests}\label{sec_np}
Asymptotic Neyman–Pearson tests are formulated such that one of the error exponents attains its maximum possible decay rate, whereas the other, unavoidably, decays at the minimal rate. Let us consider the log-likelihood ratio test comparing $X_k:=\log{l}(Y_k)$ to a threshold $t$, where $l=g_1/g_0$. For an NP-test of Type-I the threshold is chosen as $$t_0=\lim_{\epsilon\rightarrow 0}\mathbb{E}_{G_0}[\log{l}(Y_1)]+\epsilon$$ such that $P_F$ is asymptotically guaranteed to get below any $\epsilon>0$, while $P_M$ has the highest decay rate. Similarly, for an NP-test of Type-II the threshold is chosen as $$t_1=\lim_{\epsilon\rightarrow 0}\mathbb{E}_{G_1}[\log{l}(Y_1)]-\epsilon$$ such that $P_M$ is asymptotically guaranteed to get below any $\epsilon>0$, while $P_F$ has the highest decay rate. Taking the derivative of the expression inside the supremum in \eqref{eq28} with respect to \(u\) and equating it to zero yields the stationarity condition for the corresponding rate function \(I_j(t)\). 
In particular, under \(\mathcal{H}_j\), the choice \(t=t_j\) implies \(u\to 0\). Similarly, under \(\mathcal{H}_0\), \(t=t_1\) implies \(u\to 1\), while under \(\mathcal{H}_1\), \(t=t_0\) implies \(u\to -1\). Hence, for the NP-test of Type-I one obtains \(I_0(t_0)=0\) and \(I_1(t_0)=-t_0=D_{\mathrm{KL}}(G_0,G_1)\), while for the NP-test of Type-II one obtains \(I_1(t_1)=0\) and \(I_0(t_1)=t_1=D_{\mathrm{KL}}(G_1,G_0)\). Consequently, asymptotically minimax robust NP-tests can be obtained by solving
\begin{equation}\label{eq60}
    \begin{aligned}[b]
        \text{Type-I NP-test:}\quad &
        \begin{aligned}[t]
            &\min_{(G_0,G_1)\in\mathscr{G}_0\times \mathscr{G}_1}D_{\mathrm{KL}}(G_0,G_1)
             \quad \text{s.t. $g_0>0$, $g_1>0$, $\Upsilon(G_0)=1$, $\Upsilon(G_1)=1,$}\\
        \end{aligned}\\
        \text{Type-II NP-test:}\quad &
               \begin{aligned}[t]
          &\min_{(G_0,G_1)\in\mathscr{G}_0\times \mathscr{G}_1}D_{\mathrm{KL}}(G_1,G_0)
             \quad \text{s.t. $g_0>0$, $g_1>0$, $\Upsilon(G_0)=1$, $\Upsilon(G_1)=1.$}
        \end{aligned}
    \end{aligned}
\end{equation}
It is worth noting that thresholds $t_0$ and $t_1$ correspond to the two extreme values permitted by \eqref{eq26}, whereas the Bayesian test utilizes the mid-point $t=0$.
\subsection{Saddle Value Condition}\label{sec5_4}
In this section existence of a saddle value, hence a solution to the minimax optimization problem in \eqref{eq42}, is discussed. The uniqueness of this solution depends on the choice of the uncertainty classes and will be addressed in the subsequent section. In general the existence of a saddle value is described by a solution to
\begin{equation}\label{eq45}
\min_{u\in(0,1)}\sup_{(G_0,G_1)\in \mathscr{G}_0\times \mathscr{G}_1} D_u(G_0,G_1)=\sup_{(G_0,G_1)\in \mathscr{G}_0\times \mathscr{G}_1} \min_{u\in(0,1)}D_u(G_0,G_1).
\end{equation}
The equality in \eqref{eq45}, which generally holds as an inequality ($\geq$), cannot be assumed a priori and requires a rigorous analysis of both the objective function and the uncertainty sets. This can be established, for instance, by leveraging Sion’s minimax theorem \cite{sion}.As shown in Appendix~\ref{appendix2}, an application of this theorem guarantees the existence of a saddle value for \eqref{eq45}, leading to
\begin{equation}\label{eq46}
D_{\hat u}(G_0,G_1)\leq D_{\hat u}(\hat{G}_0,\hat{G}_1)\leq D_u(\hat{G}_0,\hat{G}_1),
\end{equation}
where $\hat u$ is the minimizing $u$, and $\hat{G}_0$ and $\hat{G}_1$ are the least favorable distributions.
\subsection{Problem Statement}\label{sec5_5}
From \eqref{eq46}, given $(\hat{G}_0,\hat{G}_1)$, the objective function $D_u$ needs to be minimized over $u$ and given the minimizing $\hat u$, $D_{\hat u}$ needs to be maximized over $(G_0,G_1)$. This can compactly be written as
\begin{equation}\label{eq47}
    \begin{aligned}[b]
        \text{Maximization:}\quad &
        \begin{aligned}[t]
            &\hat{g}_0=\mathrm{arg}\sup_{G_0\in\mathscr{G}_0}D_{u}(G_0,G_1)
             \quad \text{s.t. $g_0>0$, $\Upsilon(G_0)=\int_{\Omega}g_0\, d\mu=1$}\\
            &\hat{g}_1=\mathrm{arg}\sup_{G_1\in\mathscr{G}_1}D_{u}(G_0,G_1)
             \quad \text{s.t. $g_1>0$, $\Upsilon(G_1)=\int_{\Omega}g_1\,d\mu=1$}
        \end{aligned}
        \\[12pt]
        \text{Minimization:}\quad &\hat u=\mathrm{arg}\min_{u\in  (0,1)}D_{u}(\hat{G}_0,\hat{G}_1).
    \end{aligned}
\end{equation}
The maximization stage involves two separate constrained optimization problems, which are coupled. The following minimization problem can be solved once $\hat{g}_0$ and $\hat{g}_1$ are derived as functions of $u$.

\section{Derivation of Asymptotically Minimax Robust Tests}\label{sec6}
In this section LFDs and the asymptotically minimax robust tests are derived for various uncertainty classes considering the minimax optimization problem given by \eqref{eq47}. Additionally, the asymptotic NP-tests are also derived. Complete derivations are carried out for the Kullback-Leibler (KL)-divergence neighborhood, and similar steps are skipped for the sake of clarity when the same procedure is repeated for the $\alpha$- and the symmetrized $\alpha$-divergences.
\subsection{KL-divergence Neighborhood}\label{sec6_1}
Consider the uncertainty classes
\begin{equation}\label{eq49}
{\mathscr{G}}_j=\{G_j:D_{\mathrm{KL}}(G_j,F_j)\leq \epsilon_j\},\quad j\in\{0,1\},
\end{equation}
which are induced by the KL-divergence
\begin{equation*}
D_{\mathrm{KL}}(G_j,F_j)=\int_{\Omega}\log\left(g_j/f_j\right)g_j d\mu,
\end{equation*}
where $F_j$ is the nominal distribution under $\mathcal{H}_j$. The KL-divergence is considered as the classical information divergence \cite{inftheory} and used in earlier works to create the uncertainty classes \cite{dabak,levy09}. It is a smooth distance, hence suitable to deal with modeling errors \cite{gul6}.

\subsubsection{Bayesian Tests}\label{bayesian_tests}
Asymptotically minimax robust tests for the KL-divergence neighborhood can be stated with the following theorem.
\begin{thm}\label{theorem02}
For the uncertainty classes given by \eqref{eq49}, the least favorable densities
\begin{align*}
\hat{g}_0&=\exp{\left[\frac{-\lambda_0-\mu_0}{\lambda_0}\right]} \exp{\left[\frac{(1-u)(\hat{g}_1/\hat{g}_0)^{u}}{\lambda_0}\right]} f_0,\\
\hat{g}_1&=\exp{\left[\frac{-\lambda_1-\mu_1}{\lambda_1}\right]} \exp{\left[\frac{u (\hat{g}_1/\hat{g}_0)^{-1+u}}{\lambda_1}\right]} f_1,
\end{align*}
with the robust likelihood ratio function
\begin{equation}\label{eq56}
\frac{\hat{g}_1}{\hat{g}_0}=\exp{\left[\frac{-\mu_1}{\lambda_1}+\frac{\mu_0}{\lambda_0}\right]}\exp{\left[\frac{u(\hat{g}_1/\hat{g}_0)^{-1+u}}{\lambda_1}+\frac{(-1+u)(\hat{g}_1/\hat{g}_0)^u}{\lambda_0}\right]}l
\end{equation}
provide a unique solution to \eqref{eq47}. Moreover, given $u$, the Lagrangian parameters $\lambda_0$ and $\lambda_1$, hence $\mu_0$ and $\mu_1$, can be obtained by solving
\begin{align}\label{eq59x3}
\int_{\Omega} r_0 \log\left(r_0/s_0\right)f_0 d\mu/s_0=\epsilon_0, \nonumber\\
\int_{\Omega} r_1 \log\left(r_1/s_1\right)f_1 d\mu/s_1=\epsilon_1, \nonumber\\
\hat{g}_1/\hat{g}_0=(r_1 s_0)/(r_0 s_1),
\end{align}
where
\begin{align*}
s_0(\lambda_0)=\int_{\Omega}r_0(\lambda_0,\hat{g}_1/\hat{g}_0)f_0 d\mu=\int_{\Omega}\exp\left[\frac{(1-u)(\hat{g}_1/\hat{g}_0)^u}{\lambda_0}\right]f_0 d\mu=\exp\left[\frac{\lambda_0+\mu_0}{\lambda_0}\right],\nonumber\\
s_1(\lambda_1)=\int_{\Omega}r_1(\lambda_1,\hat{g}_1/\hat{g}_0)f_1 d\mu=\int_{\Omega}\exp\left[\frac{u(\hat{g}_1/\hat{g}_0)^{-1+u}}{\lambda_1}\right]f_1 d\mu=\exp\left[\frac{\lambda_1+\mu_1}{\lambda_1}\right].
\end{align*}
\end{thm}
\begin{IEEEproof}
Consider the Lagrangian
\begin{equation}\label{eq50}
L_0(g_0,g_1,\lambda_0,\mu_0)=D_{u}(G_0,G_1)+\lambda_0(\epsilon_0-D_{\mathrm{KL}}(G_0,F_0))+\mu_0(1-\Upsilon(G_0)),
\end{equation}
where $\lambda_0$ and $\mu_0$ are the KKT multipliers. A solution to \eqref{eq50} can uniquely be determined, in case all KKT conditions are met \cite[Chapter 5]{bertsekas2003convex}, because $L_0$ is a  strictly concave functional of $g_0$, as $\partial^2 L_0/\partial g_0^2<0$ for every $\lambda_0>0$. Writing \eqref{eq50} explicitly, it follows that
\begin{equation}\label{eq51}
L_0(g_0,g_1,\lambda_0,\mu_0)=\int_{\Omega}\left[\left(\frac{g_1}{g_0}\right)^u -\lambda_0\log\left(\frac{g_0}{f_0}\right)-\mu_0\right]g_0 d\mu+\lambda_0\epsilon_0+\mu_0.
\end{equation}
Imposing the stationarity condition of KKT multipliers and hereby taking the G$\hat{\mbox{a}}$teaux's derivative of Equation~\eqref{eq51} in the direction of $\psi_0$, yields
\begin{equation*}
\int_{\Omega}\left[(1-u)\left(\frac{g_1}{g_0}\right)^u-\lambda_0\log\left(\frac{g_0}{f_0}\right)-\lambda_0-\mu_0\right]\psi_0 d\mu,
\end{equation*}
which implies
\begin{equation}\label{eq53}
(1-u){g_0}^{-u}{g_1}^{u}-\lambda_0\log g_0=\lambda_0+\mu_0-\lambda_0\log f_0,
\end{equation}
since $\psi_0$ is an arbitrary function. Similarly, taking the G$\hat{\mbox{a}}$teaux's derivative of
\begin{equation}\label{eq50x}
L_1(g_0,g_1,\lambda_1,\mu_1)=D_{u}(G_0,G_1)+\lambda_1(\epsilon_1-D_{\mathrm{KL}}(G_1,F_1))+\mu_1(1-\Upsilon(G_1)),
\end{equation}
with respect to $g_1$ in the direction of $\psi_1$ leads to
\begin{equation}\label{eq54}
u{g_1}^{-1+u}{g_0}^{1-u}-\lambda_1\log g_1=\lambda_1+\mu_1-\lambda_1\log f_1.
\end{equation}
From \eqref{eq53} and \eqref{eq54}, the least favorable densities can be expressed as functionals of the robust likelihood ratio function $\hat{g}_1/\hat{g}_0$ and the nominal densities $f_0$ and $f_1$, as stated in Theorem~\ref{theorem02}. Since the second Lagrangian $L_1$ is strictly concave for every $\lambda_1 > 0$, the parametric forms of the LFDs are uniquely determined. Similarly, for any pair $(G_0, G_1) \in \mathscr{G}_0 \times \mathscr{G}_1$, $D_u$ is strictly convex in $u \in (0,1)$. This ensures that the minimizing $u$ corresponding to $\hat{g}_0$ and $\hat{g}_1$ is also unique. To determine the parameters, for any fixed value of $u$, one must solve the following system of four nonlinear equations:
\begin{align}\label{59x1}
\Upsilon(\hat{G}_0(\lambda_0,\mu_0,\lambda_1,\mu_1))&=1 \nonumber\\
\Upsilon(\hat{G}_1(\lambda_0,\mu_0,\lambda_1,\mu_1))&=1 \nonumber\\
D_{\mathrm{KL}}(\hat{G}_0(\lambda_0,\mu_0,\lambda_1,\mu_1),F_0)&=\epsilon_0 \nonumber\\
D_{\mathrm{KL}}(\hat{G}_1(\lambda_0,\mu_0,\lambda_1,\mu_1),F_1)&=\epsilon_1
\end{align}
along with the additional condition \eqref{eq56}. This system of five equations can be reduced to three without loss of generality. Specifically, from the first two equations, the parameters $s_0$ and $s_1$ can be written as functionals of $r_0$ and $r_1$, respectively, as shown in Theorem~\ref{theorem02}. Substituting $(r_0,s_0)$ and $(r_1,s_1)$ into the last two equations of \eqref{59x1} and into \eqref{eq56} yields the three-equation system given in \eqref{eq59x3}.
\end{IEEEproof}

\begin{rem}\label{rem41}
By Theorem~\ref{theorem02}, the adopted strategy proceeds by first performing the maximization step to obtain the least favorable densities in parametric form, followed by the minimization step to determine the optimal $u\in(0,1)$. Although $D_u$ is convex in $u$ for any fixed pair of known densities, this convexity does not necessarily hold in general for the parametric forms of the least favorable densities given in Theorem~\ref{theorem02}. Specifically, establishing convexity by verifying the \emph{non-negativity} of the second derivative is non-trivial, and counterexamples are readily available (see Section~\ref{sec8}). This suggests that a practical approach for determining the minimizing value of $u\in(0,1)$ is to solve the system of equations in \eqref{eq59x3} together with an additional equation that explicitly minimizes $u$.
\end{rem}

\subsubsection{Neyman-Pearson Tests}\label{sec6_np}
A solution to the Type-I NP-test formulation can be stated with the following theorem.
\begin{thm}\label{theorem03np}
The least favorable densities of the asymptotically minimax robust Type-I NP-test are given by
\begin{align}\label{eq63}
\hat{g}_0&=\left(\lambda_1+\frac{\lambda_1}{\lambda_0}\right)^{-\frac{1}{\lambda_0}}\exp{\left[-1-{\frac{1+\mu_0}{\lambda_0}}\right]}W\left[\frac{\lambda_0 \exp{\left[{\frac{-\lambda_1-\lambda_1\mu_0+\lambda_0\mu_1}{(1+\lambda_0)\lambda_1}}\right]}l^{-\frac{\lambda_0}{1+\lambda_0}}}{(1+\lambda_0)\lambda_1}\right]^{-\frac{1}{\lambda_0}}f_0\nonumber\\
\hat{g}_1&=\left(\lambda_1+\frac{\lambda_1}{\lambda_0}\right)^{-\frac{1+\lambda_0}{\lambda_0}}\exp{\left[-1-{\frac{1+\mu_0}{\lambda_0}}\right]}W\left[\frac{\lambda_0 \exp{\left[{\frac{-\lambda_1-\lambda_1\mu_0+\lambda_0\mu_1}{(1+\lambda_0)\lambda_1}}\right]}l^{-\frac{\lambda_0}{1+\lambda_0}}}{(1+\lambda_0)\lambda_1}\right]^{-\frac{1+\lambda_0}{\lambda_0}}f_0
\end{align}
where $W$ is the Lambert-W function.
\end{thm}
\begin{IEEEproof}
The solution can again be obtained by KKT multipliers. Considering the Lagrangians
\begin{align}
L_0(g_0,g_1,\lambda_0,\mu_0)=D_{\mathrm{KL}}(G_0,G_1)+\lambda_0(D_{\mathrm{KL}}(G_0,F_0)-\epsilon_0)+\mu_0(\Upsilon(G_0)-1),\nonumber\\
L_1(g_0,g_1,\lambda_1,\mu_1)=D_{\mathrm{KL}}(G_0,G_1)+\lambda_1(D_{\mathrm{KL}}(G_1,F_1)-\epsilon_1)+\mu_1(\Upsilon(G_1)-1),
\end{align}
and following the same steps as before, one can get, respectively,
\begin{align}\label{eq62}
g_1&= \exp{\left[{1+\lambda_0+\mu_0}\right]}g_0^{1+\lambda_0}{f_0}^{-\lambda_0},\\\label{eq62x1}
g_0&= g_1\left(\mu_1+\lambda_1\left(1+\log\left(g_1/f_1\right)\right)\right).
\end{align}
Solving \eqref{eq62} and \eqref{eq62x1} for $g_0$ and $g_1$, respectively, the least favorable densities of the asymptotically minimax robust NP-test of Type-I can be obtained as given in Theorem~\ref{theorem03np}.
\end{IEEEproof}

\begin{rem}
The Type-II minimax robust NP-test can similarly be obtained either by following the same procedure for the objective function $D_{\mathrm{KL}}(G_1,G_0)$ or by considering the same equations given by \eqref{eq63}. To accomplish the latter, prior to the optimization, $\epsilon_0$ and $f_0$ must be interchanged with $\epsilon_1$ and $f_1$, respectively, and after obtaining the least favorable densities, $\hat{g}_0$ must be interchanged with $\hat{g}_1$, cf.~\eqref{eq60}. The associated parameters can then be obtained directly by solving the equations in \eqref{59x1} for the least favorable densities given in \eqref{eq63}.\\
Both NP-tests are in fact limiting cases of the Bayesian asymptotically minimax robust test
\begin{equation*}
\max_{(G_0,G_1)\in\mathscr{G}_0\times \mathscr{G}_1} D_u(G_0,G_1) \equiv \min_{(G_0,G_1)\in\mathscr{G}_0\times \mathscr{G}_1} D_\alpha(G_0,G_1), \quad \forall u = \alpha \in (0,1),
\end{equation*}
where $D_\alpha(G_0,G_1)$ converges to $D_{\mathrm{KL}}(G_0,G_1)$ and $D_{\mathrm{KL}}(G_1,G_0)$, respectively, as $\alpha \rightarrow 1$ and $\alpha \rightarrow 0$, by the seventh property of $D_u$ (see also Section~\ref{sec_alpha}). This identity underscores the central role of \( D_u \), which provides a unified framework enabling the derivation of both Bayesian and Neyman-Pearson robust tests through suitable choices of $u$.

\end{rem}
\noindent Interestingly, $\hat{l} = \hat{g}_1/\hat{g}_0$ is a nonlinear functional of $l$ through $W$. Moreover, in contrast to the Bayesian asymptotically minimax robust tests, the least favorable densities corresponding to their Neyman-Pearson (NP) counterparts are expressed solely as functionals of the nominal distributions, without any coupling to $\hat{l}$. While this simplification avoids the dependence on the robust likelihood ratio, it complicates the analytic structure of the resulting least favorable densities.\\ 
Note that the problem formulation in \eqref{eq60} differs from Dabak’s approach \cite{dabak,dabak2}; see also \cite[pp.~250--255]{levy}. Specifically, Dabak’s test results from a joint minimization of $I_0(t_1)$ over all $G_1 \in \mathscr{G}_1$, and $I_1(t_0)$ over all $G_0 \in \mathscr{G}_0$. This yields simpler closed-form expressions for the least favorable densities, but the resulting test is not asymptotically minimax robust in the sense considered here (see Section~\ref{sec8}). Nonetheless, Dabak’s test turns out to be asymptotically minimax robust with respect to the expected number of samples for the sequential probability ratio test (SPRT)\cite{gulbook,gul6}.

\subsection{$\alpha-$divergence Neighborhood}\label{sec_alpha}
A natural generalization of the Kullback--Leibler divergence-based uncertainty classes is obtained by employing the $\alpha$-divergence,
\begin{equation*}
D_\alpha(G_j,F_j) := \frac{1}{\alpha(1-\alpha)} \left( \int_{\Omega} \left((1-\alpha)f_j + \alpha g_j - g_j^\alpha f_j^{1-\alpha} \right) d\mu \right), \quad \alpha \in \Omega \setminus \{0,1\},
\end{equation*}
which belongs to the broader class of $f$-divergences and subsumes various divergence measures as special cases \cite{liese87},\cite[p.~1537]{entropy}; for instance, $D_\alpha$ reduces to the Kullback--Leibler divergence $D_{\mathrm{KL}}$ as $\alpha \rightarrow 1$. The least favorable distributions associated with the $\alpha$-divergence neighborhood are characterized in the following theorem.

\begin{thm}\label{theorem04}
The least favorable densities for the $\alpha-$divergence-based uncertainty classes are given by
\begin{align}\label{eq67}
\hat{g}_0=&\left(\frac{1-\alpha}{\lambda_0}\left(\mu_0-(1-u)\left(\frac{\hat{g}_1}{\hat{g}_0}\right)^u\right)+1\right)^{\frac{1}{\alpha-1}}f_0,\nonumber\\
\hat{g}_1=&\left(\frac{1-\alpha}{\lambda_1}\left(\mu_1-u\left(\frac{\hat{g}_1}{\hat{g}_0}\right)^{-1+u}\right)+1\right)^{\frac{1}{\alpha-1}}f_1,
\end{align}
where
\begin{equation}\label{eq68}
\hat{l}=\left(\frac{\frac{1-\alpha}{\lambda_1}\left(\mu_1-u\left(\frac{\hat{g}_1}{\hat{g}_0}\right)^{-1+u}\right)+1}{\frac{1-\alpha}{\lambda_0}\left(\mu_0-(1-u)\left(\frac{g_1}{g_0}\right)^u\right)+1}\right)^{\frac{1}{\alpha-1}}l.
\end{equation}
\end{thm}
\begin{IEEEproof}
The proof follows by using the same Lagrangian approach as before, i.e. by replacing $D_{\mathrm{KL}}$ with $D_\alpha$ in \eqref{eq50} and \eqref{eq50x} and performing the derivations.
\end{IEEEproof}
The parameters are obtained similarly by solving four non-linear equations coupled with \eqref{eq68}.
\subsubsection{Special Cases}
The LFDs in \eqref{eq67} can explicitly be written for some special choices of the parameters. For instance if $\alpha=1/2$ and $u=1/2$, the robust likelihood ratio function simplifies to
\begin{equation*}
\hat{l}=\sum_{k=0}^2 c_k l^{\frac{k}{2}} 
\end{equation*}
where
\begin{align*}
c_0=\frac{0.25{\lambda_0}^2}{4{\lambda_0}^2{\lambda_1}^2+{\lambda_0}^2{\mu_0}^2+4{\lambda_0}^2{\lambda_1}\mu_0},\nonumber\\
c_1=\frac{\lambda_0(2\lambda_0\lambda_1+\mu_0\lambda_1)}{4{\lambda_0}^2{\lambda_1}^2+{\lambda_0}^2{\mu_0}^2+4{\lambda_0}^2{\lambda_1}\mu_0},\nonumber\\
c_2=\frac{2\lambda_0\lambda_1+\mu_0\lambda_1}{4{\lambda_0}^2{\lambda_1}^2+{\lambda_0}^2{\mu_0}^2+4{\lambda_0}^2{\lambda_1}\mu_0}.
\end{align*}
\subsection{Symmetric $\alpha-$divergence Neighborhood}
The $\alpha$-divergence is generally not a symmetric measure of dissimilarity, with the notable exception at $\alpha=1/2$. A symmetrized version of the $\alpha-$divergence,
\begin{equation*}
D_\alpha^s(G_j,F_j)=\frac{1}{\alpha(1-\alpha)}\left(\int_{\Omega}((f_j^\alpha -g_j^\alpha)(f_j^{1-\alpha}-g_j^{1-\alpha})) d \mu\right),\quad\alpha\in\Omega\backslash \{0,1\}
\end{equation*}
can be obtained by
\begin{equation*}
D_\alpha^s(G_j,F_j)=D_\alpha(G_j,F_j)+D_\alpha(F_j,G_j).
\end{equation*}
The symmetric $\alpha$-divergence is a member of the class of $f$-divergences~\cite{gulbook, sharp}, encompassing several well-known symmetric measures such as the symmetric Kullback–Leibler (KL) divergence and the symmetric $\chi^2$ divergence~\cite{entropy}. The least favorable distributions under the symmetric $\alpha$-divergence neighborhood are characterized by the following theorem.
\begin{thm}\label{theorem04s}
The least favorable densities under the symmetric $\alpha$-divergence neighborhood are determined by the system of coupled nonlinear equations,
\begin{align}\label{eq74}
&\frac{\lambda_0}{1-\alpha}\left(\frac{\hat{g}_0}{f_0}\right)^{2\alpha-1}+\left((1-u)\left(\frac{\hat{g}_1}{\hat{g}_0}\right)^{u}-\frac{\lambda_0}{\alpha(1-\alpha)}-\mu_0\right)\left(\frac{\hat{g}_0}{f_0}\right)^{\alpha}+\frac{\lambda_0}{\alpha}=0,\\\label{eq74x}
&\frac{\lambda_1}{1-\alpha}\left(\frac{\hat{g}_1}{f_1}\right)^{2\alpha-1}+\left(u\left(\frac{\hat{g}_1}{\hat{g}_0}\right)^{u-1}-\frac{\lambda_1}{\alpha(1-\alpha)}-\mu_1\right)\left(\frac{\hat{g}_1}{f_1}\right)^{\alpha}+\frac{\lambda_1}{\alpha}=0.
\end{align}
\end{thm}
\begin{IEEEproof}
The proof follows by applying the Lagrangian optimization framework analogously to previous derivations.
\end{IEEEproof}
In general, solving \eqref{eq74} and \eqref{eq74x} requires jointly handling four additional nonlinear equations derived from the corresponding Lagrangian constraints, resulting in a system of six equations. However, if $\alpha$ is fixed, this number can be reduced to five. The idea is to express the ratios $\hat{g}_0 / f_0$ and $\hat{g}_1 / f_1$ as functions of the likelihood ratio $\hat{g}_1 / \hat{g}_0$, i.e.,
\begin{equation}
\frac{\hat{g}_0}{f_0} = h_0\left( \frac{\hat{g}_1}{\hat{g}_0} \right), 
\quad \text{and} \quad 
\frac{\hat{g}_1}{f_1} = h_1\left( \frac{\hat{g}_1}{\hat{g}_0} \right),
\end{equation}
where $h_0$ and $h_1$ are some functions. This leads to the coupling relation
\begin{equation}
\hat{l} =  \frac{h_1(\hat{l})}{h_0(\hat{l})} l,
\end{equation}
which enables further simplification of the system.

\begin{table}[ttt]
\caption{Pair of nominal distributions used in the simulations}
\begin{center}
\begin{tabular}{|c|l|l|}
\hline
Acronym & Under $\mathcal{H}_0$ & Under $\mathcal{H}_1$ \\
\hline \hline
$d_1$ & $\mathcal{N}(-1,1)$ & $\mathcal{N}(1,1)$  \\
\hline
$d_2$ & $\mathcal{N}(-1,1)$ & $\mathcal{N}(1,4)$ \\
\hline
$d_3$ &  $\mathcal{L}(0,1)$ & $f_{\mathcal{L}}(y)(\sin(2\pi y)+1)$   \\
\hline
\end{tabular}
\end{center}
\label{tab1}
\end{table}
\section{Simulations}\label{sec8}
In this section, the theoretical findings are evaluated and exemplified. The notation $|_a^n$ stands for testing with the $(a)$-test, while the data samples are obtained from the distribuions corresponding to the $(n)$-test. In all theoretical examples, the nominal distributions listed in Table~\ref{tab1} are considered. The notation $\mathcal{N}(\mu,\sigma^2)$ stands for the Gaussian distribution with mean $\mu$ and variance $\sigma^2$, whereas $\mathcal{L}(0,1)$ denotes the standard Laplace distribution with the respective parameters. The density functions are similarly denoted by $f_{\mathcal{N}}$ and $f_\mathcal{L}$, respectively. For solving all systems of equations damped Newton's method is used \cite{ralph}. In the following, the least favorable distributions, robust likelihood ratio functions, parameters of the equations, and (non)-convexity of $D_u$ are illustrated.
\begin{figure}[ttt]
  \centering
  \centerline{\includegraphics[width=8.8cm]{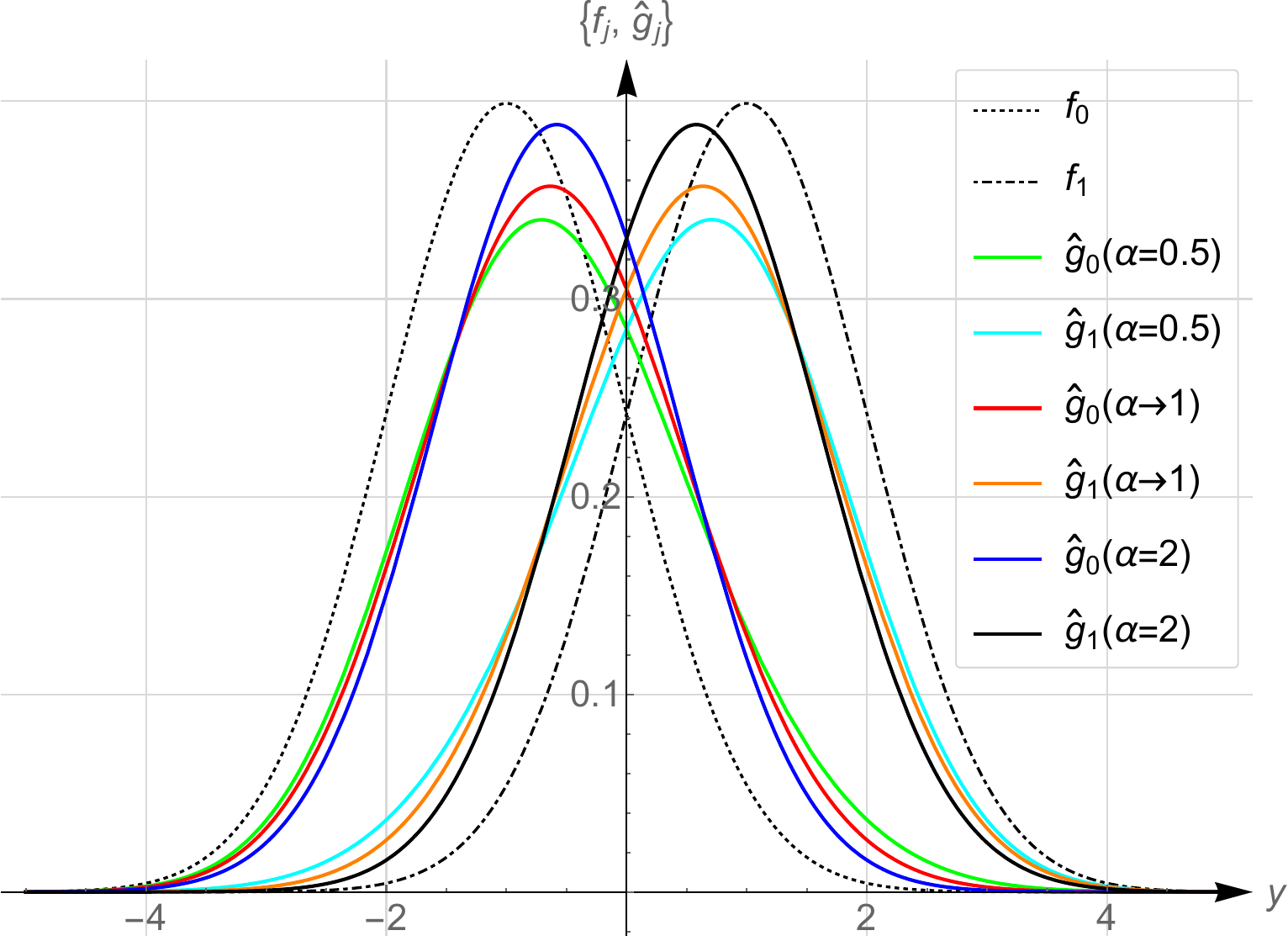}}
\caption{Nominal distributions abbreviated by $d_1$ in Table~\ref{tab1} and the corresponding LFDs with respect to their densities for $\epsilon_0=\epsilon_1=0.1$, where $u=0.5$ for all $\alpha$.\label{fig1}}
  \centering
  \centerline{\includegraphics[width=8.8cm]{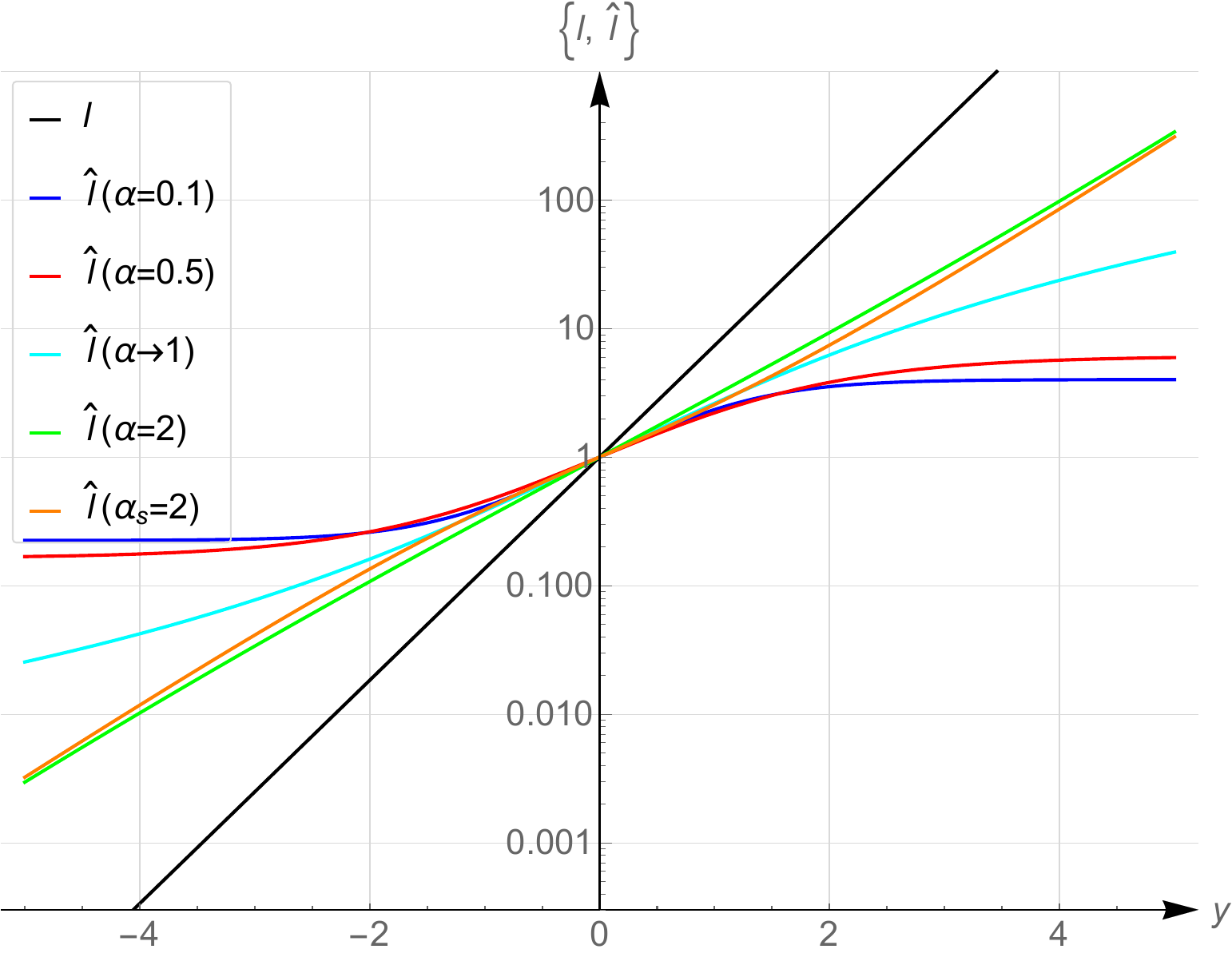}}
\caption{Robust and nominal LRFs corresponding to the nominal distributions abbreviated by $d_1$ in Table~\ref{tab1}, with $\epsilon_0=\epsilon_1=0.1$. The symmetric case $\alpha_s$ is also included.\label{fig2}}
\end{figure}
\subsection{LFDs and Robust LRFs}
The relation between the choice of the distance and its effect on robustness is investigated using the symmetric distribution pair abbreviated by $d_1$ in Table~\ref{tab1}, with robustness parameters $\epsilon_0=\epsilon_1=0.1$. For this setup, Figure~\ref{fig1} illustrates the LFDs together with the nominal distributions with respect to their densities for the KL-divergence ($\alpha\rightarrow 1$) as well as for various $\alpha$-divergences. Symmetrized $\alpha$-divergence is not included for the sake of clarity. There are two observations from this example:
\begin{enumerate}
\item The LFDs are non-Gaussian (not visible but verified by means of curve fitting).
\item The variance of the LFDs is inversely proportional to $\alpha$.
\end{enumerate}
\begin{figure}[ttt]
  \centering
  \centerline{\includegraphics[width=8.8cm]{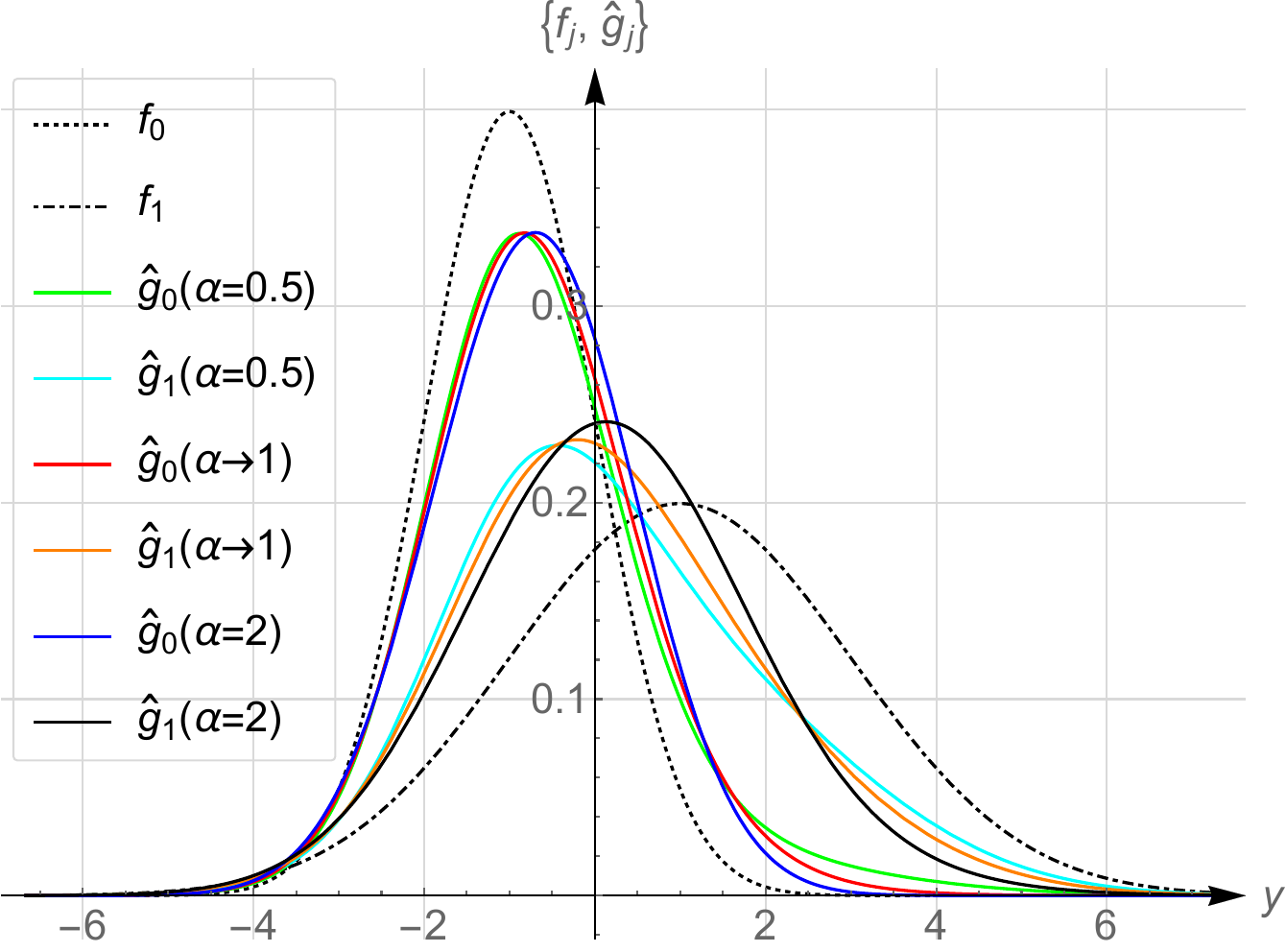}}
\caption{Nominal distributions abbreviated by $d_2$ in Table~\ref{tab1} and the corresponding LFDs for $\epsilon_0=\epsilon_1=0.1$.\label{fig3}}
  \centering
  \centerline{\includegraphics[width=8.8cm]{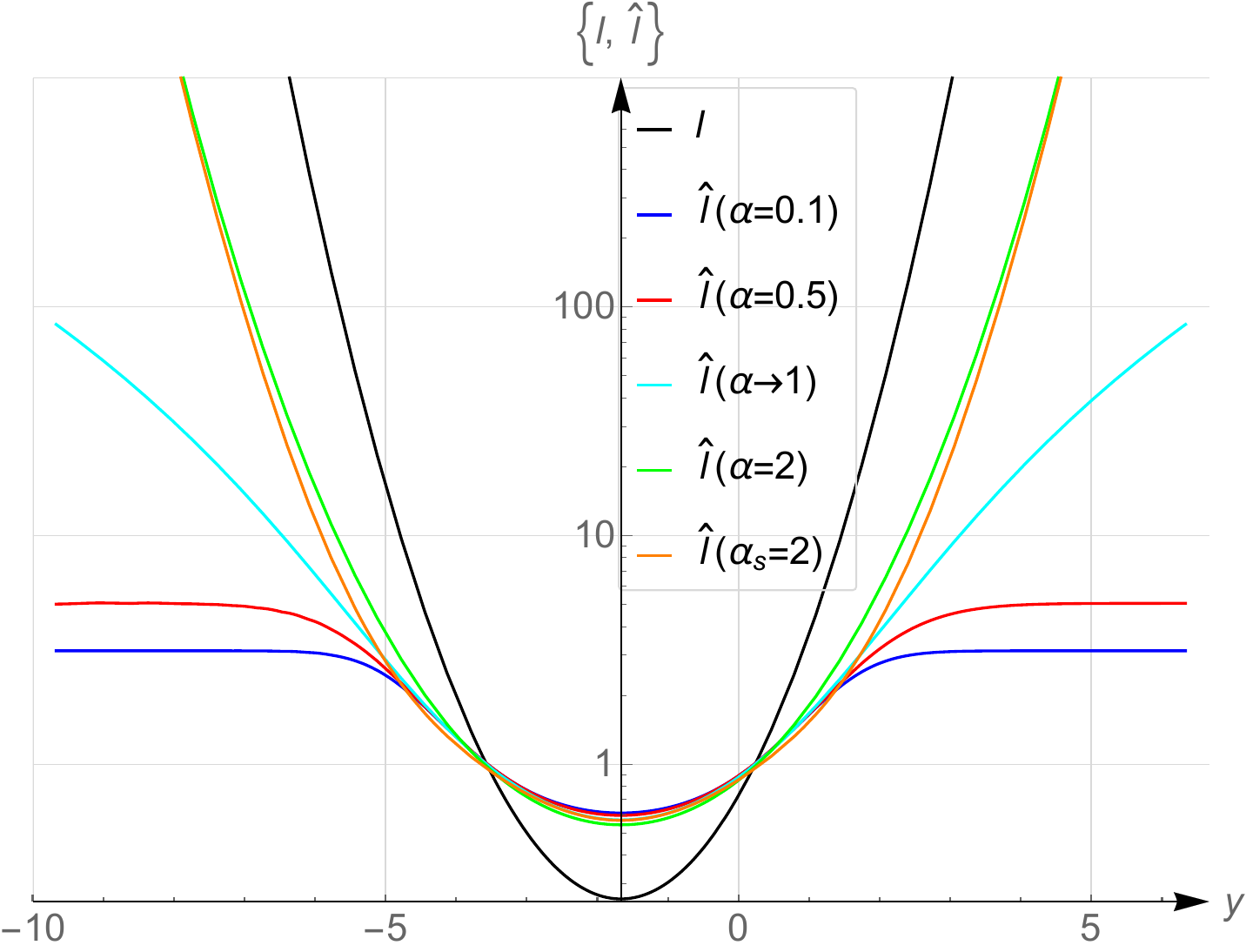}}
\caption{Robust and nominal LRFs found for the nominal distributions abbreviated by $d_2$ in Table~\ref{tab1} and $\epsilon_0=\epsilon_1=0.1$, including the symmetric case $\alpha_s$. The optimum values of $u$ are $0.95,0.67,0.56,0.59,0.61$, respectively, from $\alpha=0.1$ to $\alpha_s=2$.\label{fig4}}
\end{figure}
\noindent In Figure~\ref{fig2}, the corresponding likelihood ratio functions are illustrated, including the symmetrized version of the $\alpha$-divergence with $\alpha=2$, denoted by $\alpha_s=2$. The examples demonstrate that robustness is achieved by a clockwise rotation of the nominal LRFs followed by a clipping operation. Notably, the extent of rotation and clipping increases as $\alpha$ decreases. For smaller values of $\alpha$, the robust LRF resembles a smoothed variant of Huber’s clipped likelihood ratio test, where the nominal LRF is first rotated clockwise and subsequently clipped~\cite{hube65}. This particular behavior—yielding minimax robustness through such transformation—is, to the best of our knowledge, observed and reported for the first time in this work.\\
\noindent In the next example asymmetric pair of nominal distributions abbreviated by $d_2$ in Table~\ref{tab1} are considered. Figure~\ref{fig3} illustrates the LFDs together with the nominal distributions with respect to their densities, whereas Figure~\ref{fig4} shows the corresponding robust LRFs for $\epsilon_0=\epsilon_1=0.1$. One can make similar conclusions as in the previous example.\\
\noindent In Figures~\ref{fig2} and~\ref{fig4} the nominal LRFs are either increasing, or first decreasing and then increasing. It is possible to construct an example for which the nominal LRF is repeatedly increasing and decreasing. This case both confirms the solvability of the related non-linear equations and serves as an example for the convexity analysis in the next section. Let the nominal distributions be abbreviated by $d_3$ in Table~\ref{tab1}. Furthermore, let $\epsilon_0=\epsilon_1=0.05$, as the nominal distributions are now closer to each other. For this setup, Figure~\ref{fig5} and Figure~\ref{fig6} illustrate the LFDs together with the nominals with respect to their densities and the robust LRFs, respectively, for the KL-divergence neighborhood. Similar to the previous examples, the nominal LRFs which are smaller than $1$ are amplified and those larger than $1$ are attenuated. This implies that LFDs come closer to each other (w.r.t. the KL-divergence) in comparison to the nominal distributions.
\begin{figure}[ttt]
  \centering
  \centerline{\includegraphics[width=8.8cm]{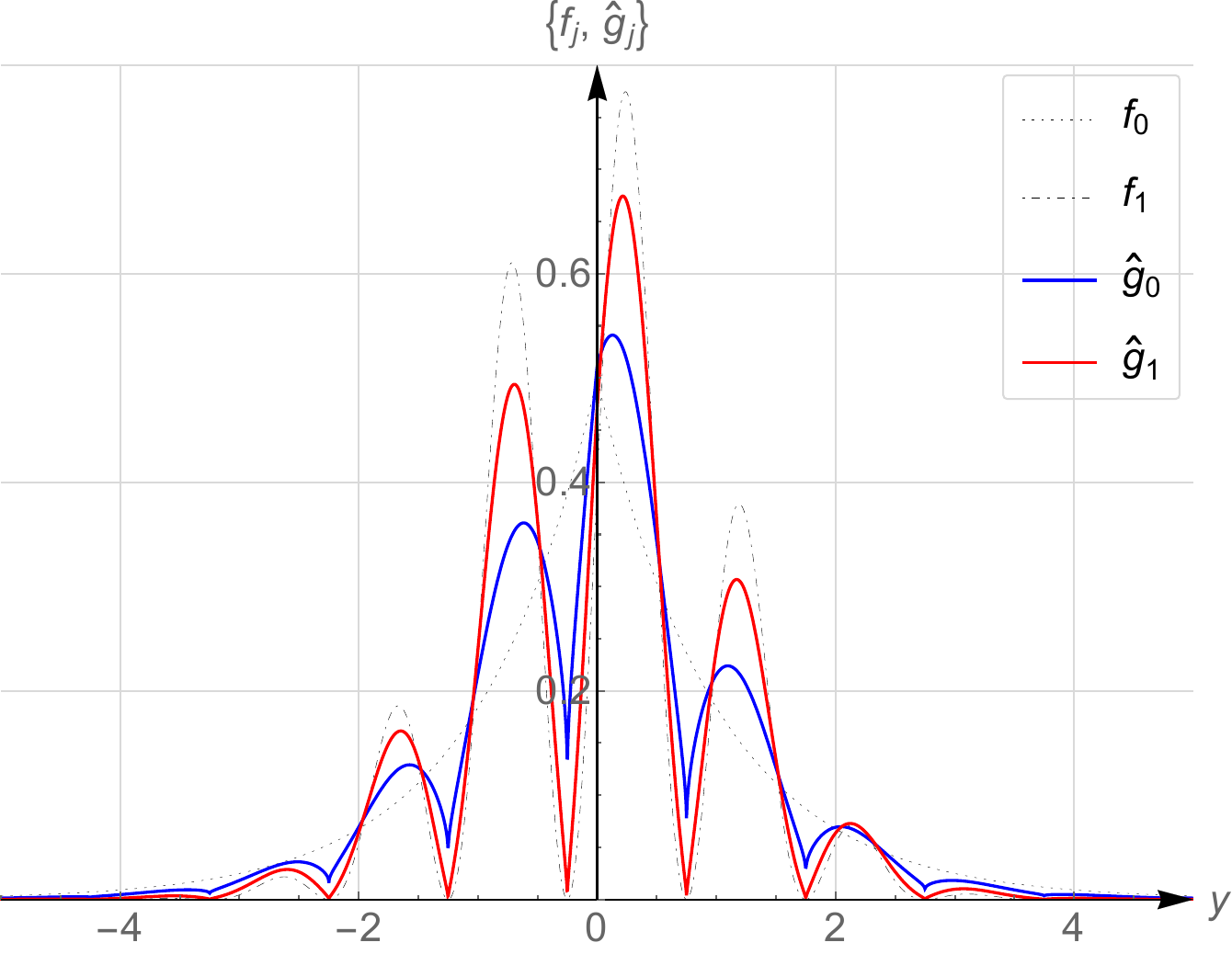}}
\caption{Nominal distributions abbreviated by $d_3$ in Table~\ref{tab1} and the corresponding LFDs with respect to their densities for $\epsilon_0=\epsilon_1=0.05$, where $u=0.46$.\label{fig5}}
  \centering
  \centerline{\includegraphics[width=8.8cm]{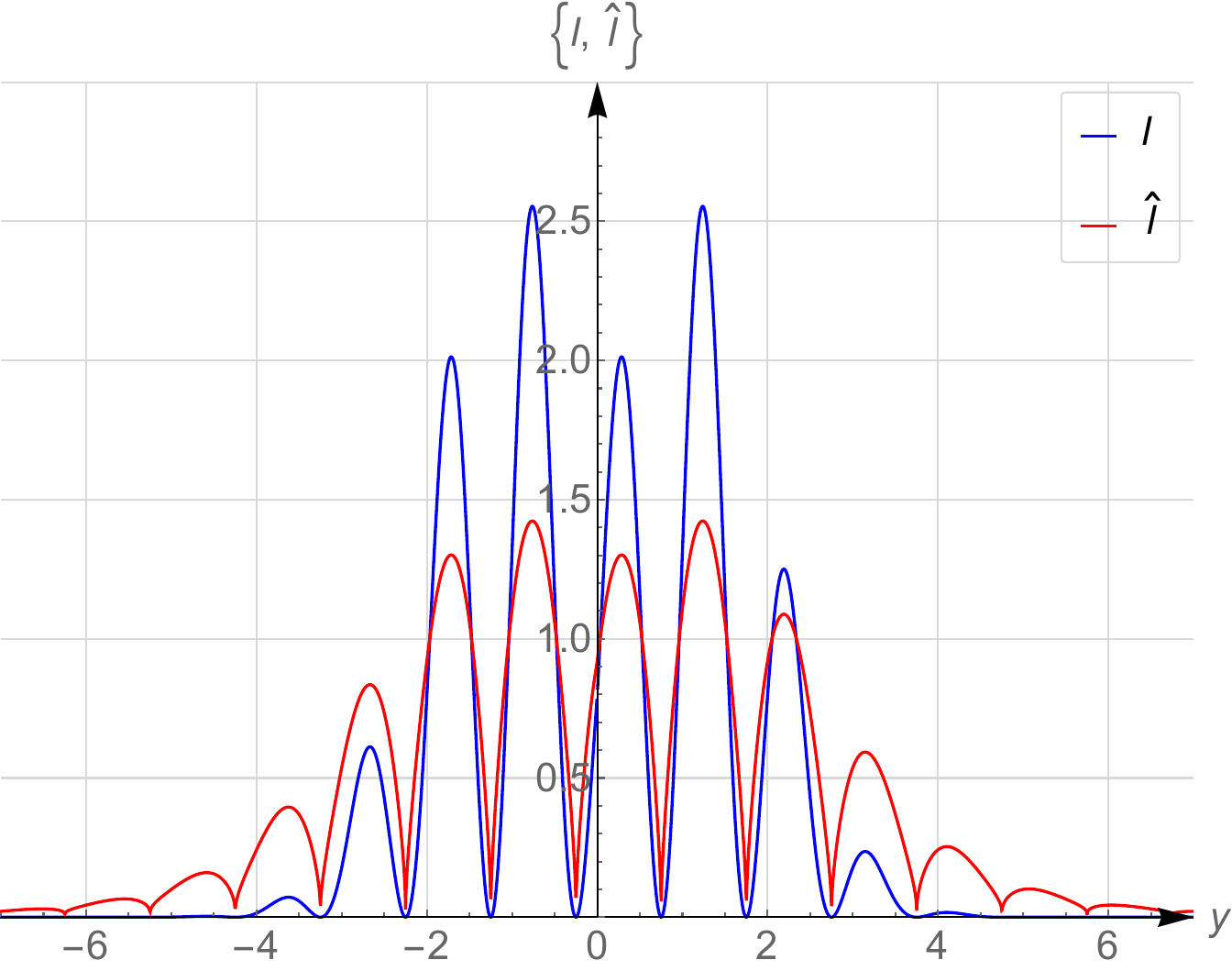}}
\caption{Robust and nominal LRFs found for the nominal distributions abbreviated by $d_3$ in Table~\ref{tab1} and $\epsilon_0=\epsilon_1=0.05$.\label{fig6}}
\end{figure}

\subsection{Convexity of $D_u$ and the Lagrangian Parameters}
As discussed in Remark~\ref{rem41}, the function $D_u$ is convex in $u$ for any fixed pair of known distributions. However, this convexity may not hold in general when the distributions themselves depend on $u$, such as in the case of least favorable distributions. In this section, we further investigate this phenomenon empirically and evaluate its implications for the minimax optimization strategy adopted in this work.\\
The primary goal here is twofold: first, to examine whether the convexity (or lack thereof) of $D_u$ can be exploited analytically or numerically during the optimization process; second, to visualize and verify the optimal Lagrangian parameters associated with the previously computed LFDs for reproducibility and transparency.\\
To this end, we proceed as follows: First, the system of nonlinear equations defined in Section~\ref{sec5} is solved for three different pairs of nominal distributions—abbreviated by $d_1$, $d_2$, and $d_3$ in Table~\ref{tab1}—under the KL-divergence neighborhood, and additionally for $d_2$ under the $\alpha$-divergence with $\alpha=0.1$, in order to obtain the corresponding LFDs. Then, for each value of $u$, the resulting LFDs are used to evaluate $D_u$, which is subsequently plotted in Figure~\ref{fig7}. These results demonstrate that $D_u$ is not necessarily convex in $u$ when the LFDs vary with $u$, thereby supporting the observation made in Remark~\ref{rem41}.\\
Furthermore, for each $u$, the optimal values of the KKT multipliers—$\lambda_0$, $\lambda_1$, $\mu_0$, and $\mu_1$—are illustrated in Figure~\ref{fig8} for the KL-divergence case with $\epsilon_0 = \epsilon_1 = 0.1$, and nominal distributions $d_1$ and $d_2$. These plots confirm that the LFDs and corresponding LRFs can be consistently reconstructed from the Lagrangian parameters, highlighting the reproducibility of the numerical procedure.\\
Beyond verification, this analysis also provides a practical justification for the sequential structure of the adopted minimax strategy—namely, performing the maximization step first (to obtain the LFDs for each $u$), followed by minimization over $u$. The observed non-convexity of $D_u$ when the LFDs depend on $u$ would make it analytically impractical to reverse this order and attempt minimization before the least favorable densities are determined.
\begin{figure}[ttt]
  \centering
  \centerline{\includegraphics[width=8.8cm]{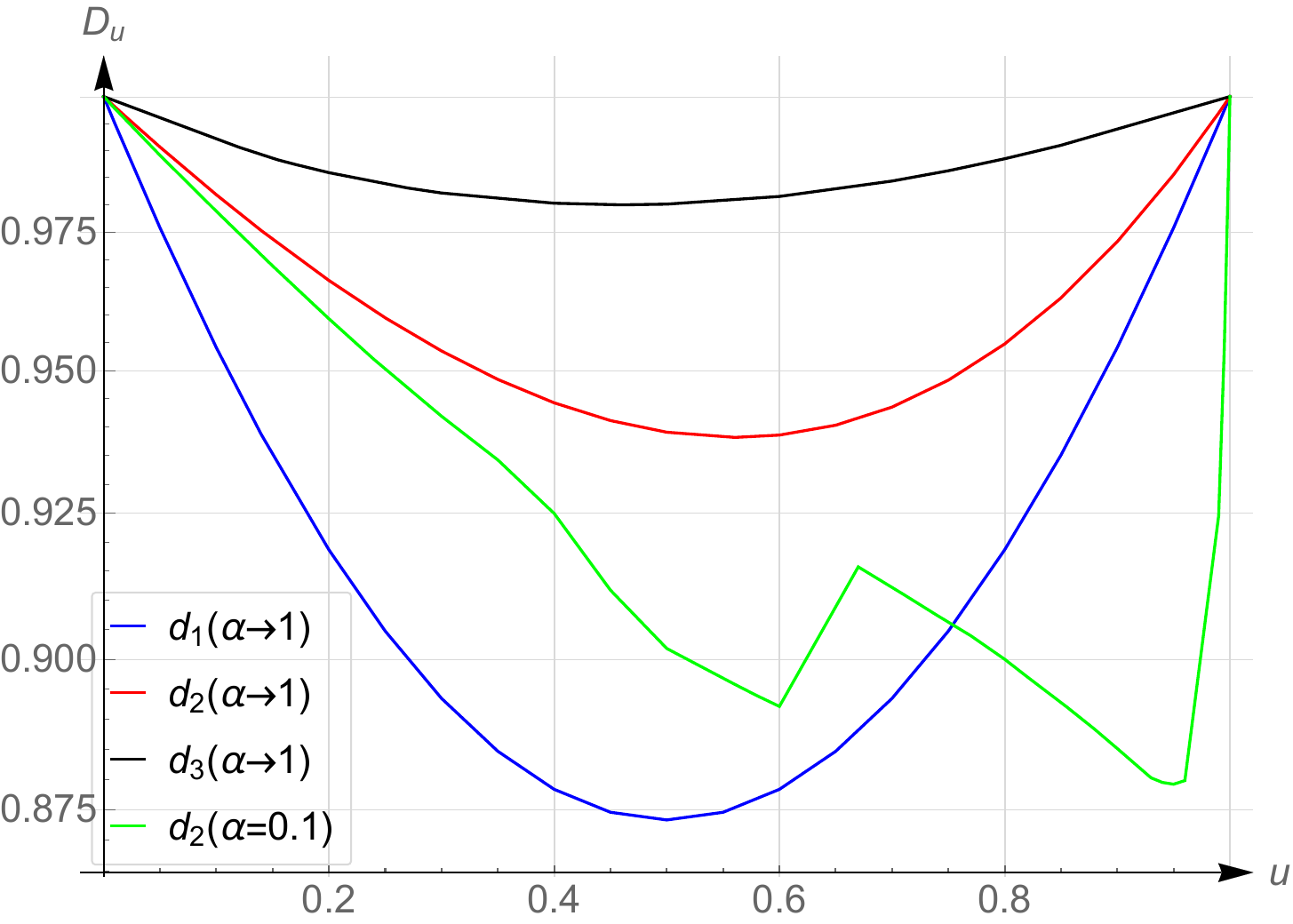}}
\caption{$u$-affinity as a function of $u$ for the LFDs obtained for various pairs of distributions as well as uncertainty classes.\label{fig7}}
  \centering
  \centerline{\includegraphics[width=8.8cm]{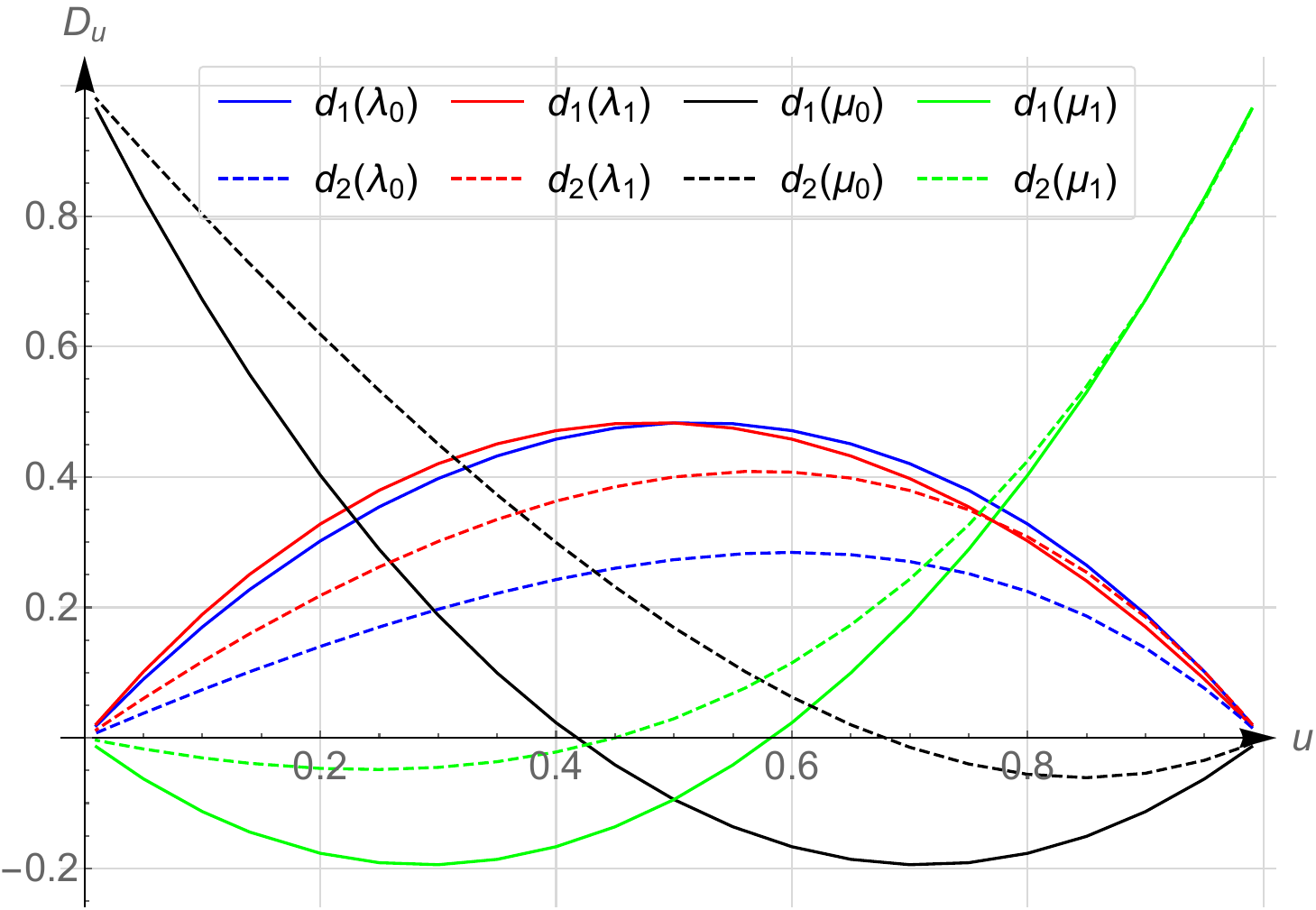}}
\caption{KKT parameters for two pairs of nominal distributions and the KL-divergence neighborhood with $\epsilon_0=\epsilon_1=0.1$.\label{fig8}}
\end{figure}

\begin{figure}[ttt]
  \centering
  \centerline{\includegraphics[width=8.8cm]{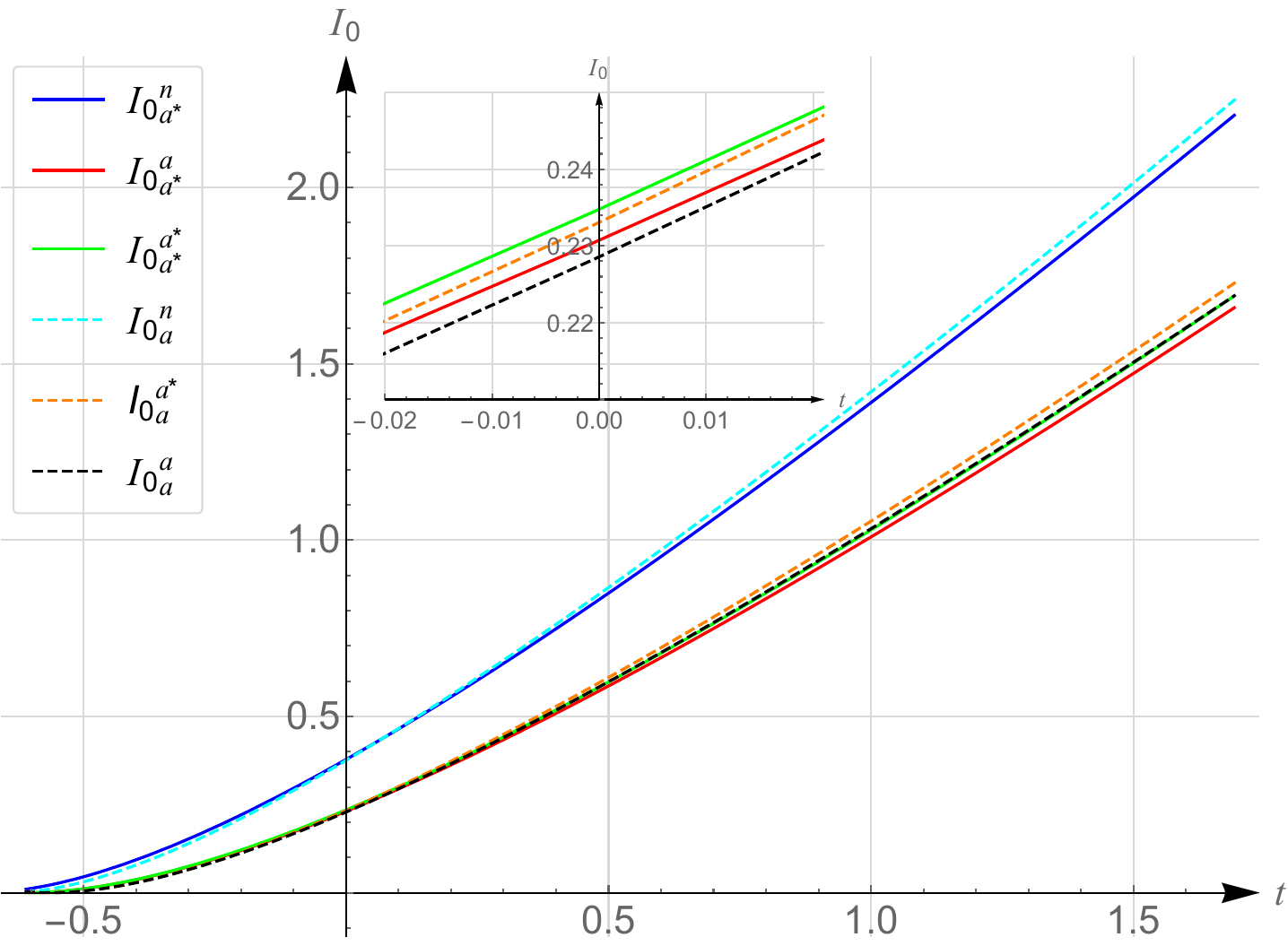}}
\caption{Asymptotic decay rate of the false alarm probability $I_0$ for the asymptotically minimax robust test ((a)-test) and Dabak's test ((a$^*$)-test).\label{fig15}}
  \centering
  \centerline{\includegraphics[width=8.8cm]{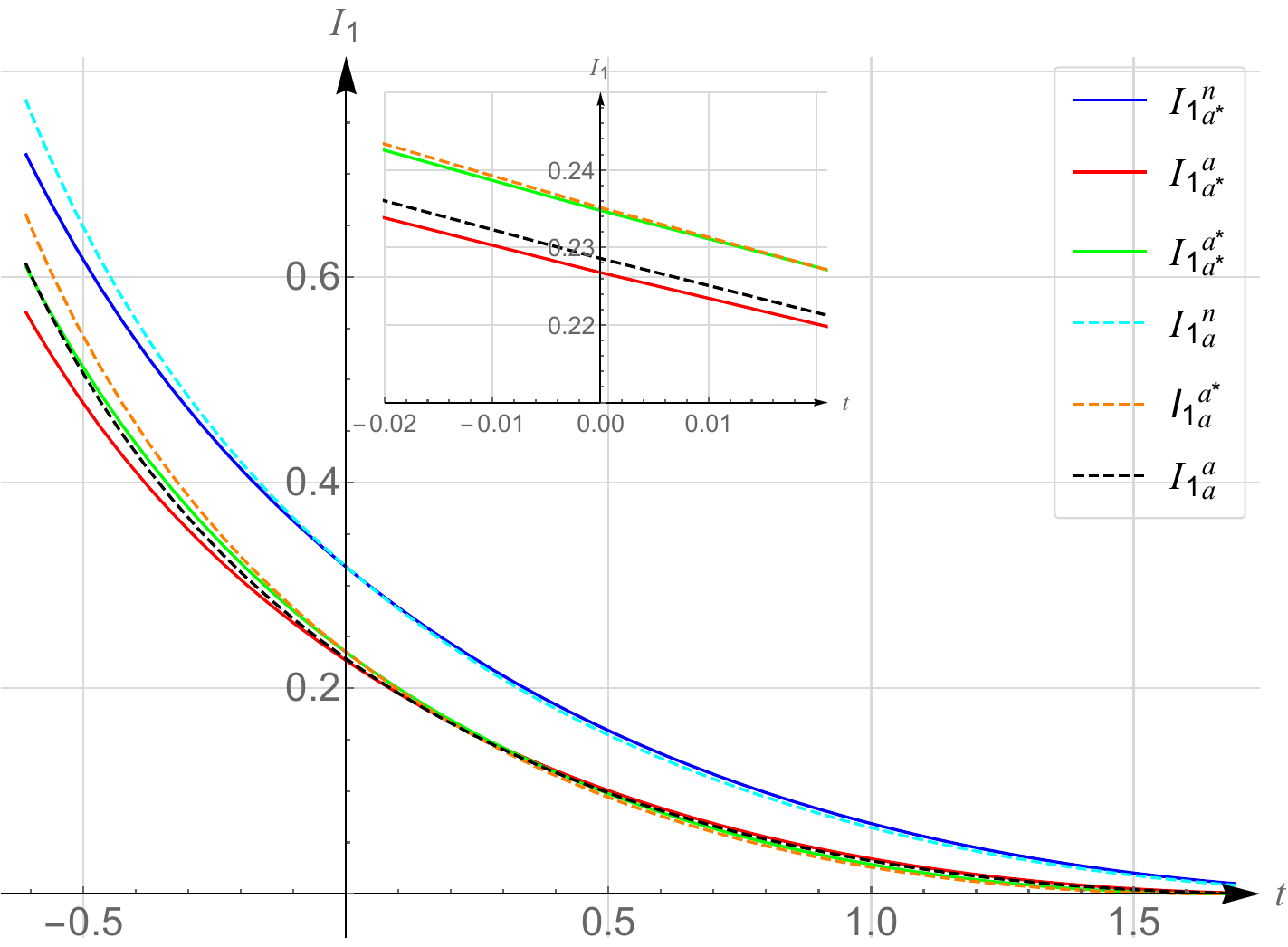}}
\caption{Asymptotic decay rate of the miss detection probability $I_1$ for the asymptotically minimax robust test ((a)-test) and Dabak's test ((a$^*$)-test).\label{fig16}}
\end{figure}
\subsection{Asymptotically Minimax Robust NP-tests}
This section presents numerical results for two tests: the Bayesian $(a)$-test, which is shown to be asymptotically minimax robust, and Dabak’s $(a^*)$-test, which is demonstrated to lack such robustness within the Neyman--Pearson (NP) framework. Data samples drawn from the distributions of the nominal test—referred to as the $(n)$-test for notational consistency—are also included to illustrate performance under distributional mismatch. These results serve to support and illustrate the theoretical robustness analysis conducted in earlier sections.\\
To illustrate this, the nominal distributions specified by row~1 in Table~\ref{tab1} are used, with uncertainty levels $\epsilon_0 = \epsilon_1 = 0.01$. Figures~\ref{fig15} and~\ref{fig16} depict the rate functions $I_0$ and $I_1$, respectively, as functions of the decision threshold $t$. For the ($a^*$)-test, only the thresholds $t = t_0 \approx -0.61$ and $t = t_1 \approx 1.69$ are relevant, as NP Type-I and Type-II tests are constructed using these thresholds (cf. Section~\ref{sec6_np}).\\
In Figure~\ref{fig15}, it is observed that $I_0|_{a^*}^{a^*}$ is not the smallest among the rate functions at $t = t_1$. Similarly, in Figure~\ref{fig16}, $I_1|_{a^*}^{a^*}$ is not the smallest at $t = t_0$. The lowest decay rates are obtained when the data is generated from the LFDs corresponding to the ($a$)-test. These results confirm that Dabak’s test is not asymptotically minimax robust in the NP formulation.\\
Since the selected setting permits the evaluation of additional thresholds, the behavior of the ($a$)-test under mismatched LFDs is also examined at $t = 0$. In the zoomed-in regions of Figures~\ref{fig15} and~\ref{fig16}, the black dashed lines correspond to the scenario where both the test and the data are generated from the LFDs of the ($a$)-test. These lines consistently exhibit the smallest error exponents. In contrast, the remaining dashed lines—corresponding to various mismatch cases—yield larger decay rates. This behavior indicates that the performance of the ($a$)-test does not degrade under mismatch, in line with its minimax robustness properties.\\

\section{Conclusion}\label{sec9}
This work has developed a unified framework for constructing asymptotically minimax robust tests under distributional uncertainty. The central optimization problem was formulated as a saddle-point problem involving the $u$-affinity, where the $u$-affinity is maximized over pairs of distributions $(G_0, G_1)$ from predefined uncertainty classes, and minimized over $u \in (0,1)$. This formulation enabled the derivation of robust tests applicable to both Bayesian and Neyman--Pearson formulations within a single theoretical structure.\\
Several $f$-divergence-based uncertainty classes were considered, including neighborhoods defined by the KL-divergence, the $\alpha$-divergence, and its symmetrized form. For these, least favorable densities were derived in closed parametric form using Karush-Kuhn-Tucker conditions, leading to tractable expressions involving, in some cases, nonlinear functions such as the Lambert-W function. The resulting robust likelihood ratio functions were shown to be nonlinear transformations—specifically, rotated and clipped versions—of their nominal counterparts, particularly under $\alpha$-divergence neighborhoods.\\
Numerical simulations supported the theoretical findings, validating the robustness of the proposed Bayesian ($a$)-test and confirming that Dabak’s ($a^*$)-test is not asymptotically minimax robust in the NP setting. In addition to verifying robustness guarantees, the simulations also provided practical insights into the behavior of the robust LRFs and the Lagrangian parameters. All results are consistent with the theoretical predictions and offer reproducible evidence of the test's minimax optimality.

\appendices\label{appx}

\section{Proof of Theorem~\ref{theorem01}}\label{appendix3}
Let us consider the error probability given by \eqref{eq2} for the minimax test \eqref{eqgfg}, but for simplicity in terms of only $n$ and $t$ as
\begin{equation*}
P_E(n,t)=\pi_0P_F(n,t)+(1-\pi_0)P_M(n,t).
\end{equation*}
From \eqref{eq27} and \eqref{eq275} one can write
\begin{equation*}
P_E(n,t)\approx C_F(n)\pi_0\exp\left(-nI_0(t)\right)+C_M(n)(1-\pi_0)\exp\left(-nI_1(t)\right),
\end{equation*}
where $C_F$ and $C_M$ satisfy
\begin{equation}\label{eq32}
\lim_{n\rightarrow\infty}\frac{1}{n}\log C_F(n)=\lim_{n\rightarrow\infty}\frac{1}{n}\log C_M(n)=0.
\end{equation}
Hence, the exponential decay rate of the error probability is governed by $I_0$ and $I_1$. For any $t$, for which $I_1(t)>I_0(t)$ we have
\begin{equation}
P_E(n,t)=\pi_0C_F(n)\exp(-nI_0(t))\,\,\text{as}\,\,n\to\infty,
\end{equation}
since
\begin{align}
\frac{P_E(n,t)}{\pi_0C_F(n)\exp(-nI_0(t))}&=\frac{\pi_0C_F(n)\exp(-nI_0(t))+(1-\pi_0)C_M(n)\exp(-nI_1(t))}{\pi_0C_F(n)\exp(-nI_0(t))}\nonumber\\
&=1+C(n)\exp(-n(I_1(t)-I_0(t)))\to 1,
\end{align}
where
\begin{equation}
C(n)=\frac{(1-\pi_0)C_M(n)}{\pi_0C_F(n)}.
\end{equation}
This argument is true because $I_1(t)-I_0(t)>0$ and $C(n)$ is sub-exponential since from \eqref{eq32} we have
\begin{equation}
\frac{1}{n}(\log((1-\pi_0)C_M(n))-\log(\pi_0C_F(n)))=0\,\,\text{as}\,\,n\to\infty.
\end{equation}
Similarly, for the case $I_1(t)<I_0(t)$ we have
\begin{equation}
P_E(n,t)=(1-\pi_0)C_M(n)\exp(-nI_1(t))\,\,\text{as}\,\,n\to\infty.
\end{equation}
Consequently, as $n\to\infty$ we have
\begin{equation}
P_E(n,t) = \begin{cases} \pi_0C_F(n)\exp(-nI_0(t)), & I_1(t)>I_0(t)  \\ (1-\pi_0)C_M(n)\exp(-nI_1(t)), & I_0(t)>I_1(t) \end{cases},
\end{equation}
which can be rewritten as
\begin{equation}
P_E(n,t) = \pi_0^a(1-\pi_0)^{1-a}C_F(n)^aC_M(n)^{1-a}\exp(-n\min\{I_0(t),I_1(t)\})
\end{equation}
where $a=\mathbf{1}_{\{I_1>I_0\}}$. Hence, as $n\to\infty$
\begin{align}
\min_t P_E(n,t) &\equiv \min_t \exp(-n\min\{I_0(t),I_1(t)\})\equiv \min_t -n\min\{I_0(t),I_1(t)\}\nonumber\\
&\equiv\max_t \min\{I_0(t),I_1(t)\}\equiv\min_t \max\{I_0(t),I_1(t)\},
\end{align}
since $\pi_0^a(1-\pi_0)^{1-a}C_F(n)^aC_M(n)^{1-a}$ is positive and independent of $t$. From \cite[Remark. 5.2.2.]{gulbook}, $I_0$ and $I_1$ are increasing and decreasing functions of $u$, respectively. Let $h_j:u\mapsto t$ be the mapping between maximizing $u$ and $t$ in \eqref{eq28}. It is easy to see that $h_j$ is increasing because it is the derivative of a convex function $\log M_{X_1}^j(u)$ \cite[p. 77]{levy}. Hence, $I_0(t)=I_0(h_0(u))$ and $I_1(t)=I_1(h_1(u))$ are also increasing and decreasing functions respectively, as
\begin{align*}
\frac{d I_0(h_0(u))}{d u}&=I^{'}_0(h_0(u))h_0^{'}(u)\geq 0, \\
\frac{d I_1(h_1(u))}{d u}&=I^{'}_1(h_1(u))h_1^{'}(u)\leq 0.
\end{align*}
Since $I_0$ and $I_1$ are also increasing and decreasing functions of $t$, and furthermore, as $M_{X_1}^1(u)=M_{X_1}^0(u+1)$ for \text{$G_j:=\hat{G}_j$} together with \eqref{eq28} implies \text{$I_1(t)=I_0(t)-t$}, it is true that \text{$I_0(0)=I_1(0)$} and together with \text{$\{t:I_1>I_0\}\equiv\{t:t<0\}$} and \text{$\{t:I_1<I_0\}\equiv\{t:t>0\}$} one can write
\begin{equation}\label{eq35}
I_m(t)=\min\{I_0(t),I_1(t)\}=
\begin{cases}
I_0(t), &  t<0 \\
I_1(t), &  t>0 \\
I_0(0)=I_1(0), &  t=0
\end{cases}.
\end{equation}
Hence, we have
\begin{equation}\label{eq36}
\arg\sup_t I_m(t)=0.
\end{equation}
Notice that we need $G_j:=\hat{G}_j$ in Theorem~\ref{theorem0}. Else, \eqref{eq35} and \eqref{eq36} do not necessarily hold.

\section{Existence of a Saddle Value for \eqref{eq45}}\label{appendix2}
\begin{thm}[Application of Sion's minimax theorem \cite{sion}]
A solution to \eqref{eq45} exists if the following conditions hold:
\begin{itemize}
\item The objective function $D_u$ is real valued, upper semi-continuous and quasi-concave on ${\mathscr{G}}_0\times{\mathscr{G}}_1$ for all $u\in[0,1]$.5
\item The objective function $D_u$ is lower semi-continuous and quasi-convex on $[0,1]$ for all $(G_0,G_1)\in{\mathscr{G}}_0\times{\mathscr{G}}_1$.
\item $[0,1]$ is a compact convex subset of a linear topological space.
\item ${\mathscr{G}}_0\times{\mathscr{G}}_1$ is a convex subset of a linear topological space.
\end{itemize}
\end{thm}

\begin{IEEEproof}
The objective function is real valued, continuous in $u$ and $(g_0,g_1)$, jointly concave on ${\mathscr{G}}_0\times{\mathscr{G}}_1$ for all $u\in[0,1]$, and convex on $[0,1]$ for all $(G_0,G_1)\in{\mathscr{G}}_0\times{\mathscr{G}}_1$, see $2$nd and $3$rd properties of $D_u$. The set $[0,1]$ is trivially convex and is closed and bounded, hence compact with respect to the standard topology by Heine-Borel theorem \cite[Theorem 2.41]{rudin1976}. Finally, ${\mathscr{G}}_0$ and ${\mathscr{G}}_1$ are convex sets, since $D_f$ is a convex distance. As a result ${\mathscr{G}}_0\times {\mathscr{G}}_1$ is also convex.
\end{IEEEproof}

\bibliographystyle{IEEEtran}
\bibliography{strings4}
\end{document}